\documentclass[sn-mathphys]{sn-jnl}

\usepackage{amsmath,amssymb,bm}
\usepackage{xspace}
\usepackage{cleveref}
\usepackage{graphics}
\usepackage{epsfig}
\usepackage{caption}
\usepackage{subcaption}
\usepackage{algorithm}
\usepackage{algpseudocode}

\theoremstyle{thmstyleone}%
\newtheorem{theorem}{Theorem}
\theoremstyle{thmstyletwo}%
\newtheorem{remark}{Remark}%

\theoremstyle{thmstylethree}%

\begin{document}

\title[Convergence Analysis of Dirichlet Energy Minimization]{Convergence Analysis of Dirichlet Energy Minimization for Spherical Conformal Parameterizations}


\author[1]{\fnm{Wei-Hung} \sur{Liao}}\email{roger2300245@gmail.com}

\author*[2]{\fnm{Tsung-Ming} \sur{Huang}}\email{min@ntnu.edu.tw}

\author[3,1]{\fnm{Wen-Wei} \sur{Lin}}\email{wwlin@math.nctu.edu.tw}

\author[2]{\fnm{Mei-Heng} \sur{Yueh}}\email{yue@ntnu.edu.tw}

\affil[1]{\orgdiv{Department of Applied Mathematics}, \orgname{National Yang Ming Chiao Tung University}, \orgaddress{\city{Hsinchu}, \postcode{300},  \country{Taiwan}}}

\affil*[2]{\orgdiv{Department of Mathematics}, \orgname{National Taiwan Normal University}, \orgaddress{\city{Taipei}, \postcode{116}, \country{Taiwan}}}

\affil[3]{\orgname{Nanjing Center for Applied Mathematics}, \orgaddress{\city{Nanjing}, \postcode{211135},  \country{People's Republic of China}}}


\abstract{In this paper, we first derive a theoretical basis for spherical conformal parameterizations between a simply connected closed surface $\mathcal{S}$ and a unit sphere $\mathbb{S}^2$ by minimizing the Dirichlet energy on $\overline{\mathbb{C}}$ by stereographic projection. The Dirichlet energy can be rewritten as the sum of the energies associated with the southern and northern hemispheres and can be decreased under an equivalence relation by alternatingly solving the corresponding Laplacian equations. Based on this theoretical foundation, we develop a modified Dirichlet energy minimization with nonequivalence deflation for the computation of the spherical conformal parameterization between $\mathcal{S}$ and $\mathbb{S}^2$. In addition, under some mild conditions, we verify the asymptotically R-linear convergence of the proposed algorithm. Numerical experiments on various benchmarks confirm that the assumptions for convergence always hold and indicate the efficiency, reliability and robustness of the developed modified Dirichlet energy minimization.}

\keywords{Dirichlet energy minimization, spherical conformal parameterization, nonequivalence deflation, asymptotically R-linear convergence.}



\maketitle

\section{Introduction} \label{sec:Intro}

The Poincar\'{e}--Klein--Koebe uniformization theorem \cite{koebe1907, poincare1908} claims that a simply connected Riemann surface is conformally equivalent to one of three canonical Riemann surfaces, namely, a sphere $\mathbb{S}^2=\overline{\mathbb{C}}=\mathbb{R}^2\cup\{\infty\}$, a complex plane $\mathbb{C}$, or a unit disk $\mathbb{D}$. In this paper, we focus on the study of the spherical conformal parameterization between a simply connected closed surface $\mathcal{S}$, i.e., a closed surface of genus zero, and the unit sphere $\mathbb{S}^2$. It is well known from the Dirichlet principle in \cite{hilbert1935} that a spherical conformal map from $\mathcal{S}$ to $\mathbb{S}^2$ solves the optimization problem of Dirichlet energy with constraints on $\mathbb{S}^2$. However, since the constraint of $\mathbb{S}^2$
is not a convex domain, it is generally difficult for a gradient descent projection method or a heat diffusion flow approach on $\mathbb{S}^2$ to show convergence. Therefore, in this paper, we reexamine the expression of the Dirichlet energy on $\overline{\mathbb{C}}$ and derive a theoretical foundation for Dirichlet energy minimization. On this basis, we will develop a new numerical algorithm for efficiently solving the spherical conformal map and prove that the numerical method has asymptotically R-linear convergence.

We first survey numerical methods developed in the past as follows: Because harmonic mappings between genus-zero closed surfaces are conformal \cite{hilbert1935}, the computation of the spherical conformal parameterization problem is equivalent to finding a harmonic mapping between the surface and the unit sphere. 
Angenent et al. \cite{AnHT99} first proposed a method for computing spherical harmonic maps of genus-zero closed triangular meshes by solving a linear Laplace--Beltrami equation with a 3-point boundary condition in view of the Dirichlet principle. The computational cost of \cite{AnHT99} is quite low, as the task is merely to solve a linear system with the size of the number of vertices. However, the drawback is that the conformal distortions in the region near the three constrained points are relatively large. 
Nevertheless, the map provided by \cite{AnHT99} can be generally considered a good initial map for the other existing algorithms. 

Later, Gu et al. \cite{GuWC04} proposed a nonlinear heat diffusion algorithm with a forward Euler method for the computation of spherical harmonic maps, which has highly improved conformal distortions compared to \cite{AnHT99}. However, the required number of iterations for convergence is not satisfactory, and this method usually takes a few minutes for a triangular mesh with approximately 100K vertices. Next, Huang et al. \cite{HuGH14} improved the efficiency of nonlinear heat diffusion by the quasi-implicit Euler method (QIEM), which uses a centralized and normalized projection to guarantee that the image is always on $\mathbb{S}^2$ in the iteration procedure. Because this projection is highly nonlinear, the convergence of the heat diffusion algorithms \cite{GuWC04,HuGH14} cannot be theoretically guaranteed.

A year later, Choi et al. \cite{ChLL15} proposed a fast landmark aligned spherical harmonic (FLASH) parameterization based on a quasi-conformal approach that aims to minimize the absolute values of the Beltrami coefficients. Numerical experiments indicate that the FLASH algorithm has better efficiency and fewer conformal distortions than heat diffusion algorithms \cite{GuWC04,HuGH14}. Very recently, a parallel version of FLASH developed by \cite{ChLG20} was found to significantly boost the computational cost when implemented in the environment of parallel machines.

In recent years, Yueh et al. \cite{YuLW17,YuLL19} proposed a Dirichlet energy minimization method (DEM) to alternatingly minimize the associated Dirichlet energy on $\overline{\mathbb{C}}$ corresponding to the southern and northern hemispheres by stereographic projection. Numerical experiments indicate that this method has a slightly improved efficiency and similar conformal distortion compared to the FLASH algorithms \cite{ChLL15}. Nevertheless, this approach lacks theoretical foundations for convergence.

In this paper, we first provide a theoretical foundation for the computation of the spherical conformal parameterization between a genus-zero closed surface $\mathcal{S}$ and $\mathbb{S}^2$ by minimizing the Dirichlet energy on $\overline{\mathbb{C}}$ by the flattening technique of stereographic projection from $\mathbb{S}^2$ to $\overline{\mathbb{C}}$. The Dirichlet energy can be divided into the sum of energies corresponding to the southern and northern hemispheres and then reduced under an equivalence relation by alternatingly solving the associated Laplacian equations. Based on this theoretical foundation, we propose a modified Dirichlet energy minimization (MDEM) with a nonequivalence deflation technique for computing the spherical conformal parameterization between $\mathcal{S}$ and $\mathbb{S}^2$. Furthermore, we prove the asymptotically R-linear convergence of the MDEM algorithm under some mild conditions. Numerical experiments on various models from benchmarks validate that the conditions for convergence always hold and show the efficiency and reliability of MDEM.

\noindent The main contributions of this paper are threefold.
\begin{itemize}
\item The original Dirichlet energy minimization problem \cite{hilbert1935} for spherical conformal parameterization between the genus-zero closed surface $\mathcal{S}$ and $\mathbb{S}^2$ is a 3-dimensional vector-valued variation that considers not only the tangential variation but also the normal variation, which makes the problem much more complicated. We apply a stereographic projection from $\mathbb{S}^2$ to $\overline{\mathbb{C}}$ to reduce the dimension of the variational space. Then, we propose a theoretical foundation of the DEM on $\overline{\mathbb{C}}$ for spherical conformal parameterizations. Using the calculus of variations, we derive the Euler--Lagrange equations and divide them into two parts, corresponding to the homothetic and minimal transformations. We further determine the equations that are meaningful for the angle-preserving transformation corresponding to the minimal transformation part. Then, we define an equivalence relation to compare the reduction of energy that contributes to conformality.
\item Based on the theoretical foundation above, we propose an efficient and reliable MDEM algorithm with nonequivalence deflation for the computation of spherical conformal parameterizations between a genus-zero closed triangular mesh and $\mathbb{S}^2$. We then prove that the MDEM algorithm converges asymptotically and R-linearly under certain conditions.
\item Numerical experiments on the important BraTS 2021 medical databases~\cite{BaAS17,BaGB21} with 1251 brain images validate that the assumptions for convergence hold for all brain samples. With the aid of spherical conformal parameterizations for transforming brain images to the canonical domain $\mathbb{S}^2$, brain tumor detection and segmentation, as well as image processing tasks such as alignment, registration and texture mapping, can be smoothly carried out by AI algorithms.
\end{itemize}

This paper is organized as follows: In Section \ref{sec:Thm_Found_DEM}, we introduce the theoretical foundation for the computation of spherical harmonic maps by minimizing the Dirichlet energy on the image of $\overline{\mathbb{C}}=\mathbb{R}^2\cup\{\infty\}$. In Section \ref{sec:3}, we propose a modified DEM with a nonequivalence deflation technique for the computation of spherical conformal parameterizations of simply connected closed triangular meshes and prove the asymptotically R-linear convergence of the MDEM algorithm. Numerical validations of the asymptotically R-linear convergence and the conformality of MDEM are given in Section \ref{sec:4}. Finally, a concluding remark is given in Section \ref{sec:5}.

In this paper, we use the following notation:
\begin{itemize}[label={$\bullet$}]
\item Bold letters, for instance $\mathbf{f}$, denote (complex) vectors.
\item Capital letters, for instance $A$, denote matrices.
\item Typewriter-style letters, for instance $\mathtt{I}$ and $\mathtt{B}$, denote ordered sets of indices. 
\item $\mathbf{f}_i$ denotes the $i$th entry of the vector $\mathbf{f}$.
\item $\mathbf{f}_\mathtt{I}$ denotes the subvector of $\mathbf{f}$ composed of $\mathbf{f}_i$ for $i\in\mathtt{I}$.
\item $\lvert \mathbf{f}\rvert$ denotes a vector with entries that are the absolute values of those in $\mathbf{f}$. 
\item $A_{i,j}$ denotes the $(i,j)$th entry of matrix $A$.
\item $A_{\mathtt{I},\mathtt{J}}$ denotes the submatrix of $A$ composed of $A_{i,j}$ for $i\in\mathtt{I}$ and $j\in\mathtt{J}$.
\item $\mathbb{S}^2:=\{ \mathbf{x}\in\mathbb{R}^{3} \mid \|\mathbf{x}\|=1 \}$ denotes the $2$-sphere in $\mathbb{R}^{3}$.
\item $\mathbb{D}:=\{ z\in\mathbb{C} \mid \lvert z \rvert <1  \}$ denotes the unit disk in $\mathbb{C}$.
\item $\left[{v}_0, \ldots, {v}_m\right]$ denotes the $m$-simplex of the vertices ${v}_0, \ldots, {v}_m$. 
\item $\mathrm{i}$ denotes the imaginary unit $\sqrt{-1}$. 
\end{itemize}

\section{Theoretical foundations for the DEM on $\overline{\mathbb{C}}$}
\label{sec:Thm_Found_DEM}

In this section, we propose a theoretical foundation for the computation of spherical harmonic maps by minimizing the Dirichlet energy on the image of $\overline{\mathbb{C}}=\mathbb{R}^2\cup\{\infty\}$. The well-known uniformization theorem in \cite{koebe1907, poincare1908} states that every simply connected closed Riemann surface $\mathcal{S}$ is conformally equivalent to the unit sphere $\mathbb{S}^2$. To study the structure of conformality between $\mathcal{S}$ and $\mathbb{S}^2$, the flattening technique by stereographic projection from $\mathbb{S}^2$ to $\overline{\mathbb{C}}$ is considered to make this method more feasible in terms of applications.

Let $\mathbb{D}$ be a unit disk, and let $(\mathbb{D}, \delta)$ and $(\mathbb{S}^2, \widetilde{\delta})$ be Riemann surfaces endowed with the Euclidean metrics $\delta= \delta_{\alpha \beta}dx^\alpha dx^\beta$, $\alpha, \beta = 1, 2$, and $\widetilde{\delta}=\widetilde{\delta}_{ij}df^idf^j$, $i, j= 1, 2, 3$, respectively, where $\delta_{\alpha\beta}$ and $\widetilde{\delta}_{ij}$ denote the Kronecker delta. 
The Dirichlet energy of a $C^1$-map $f:\mathbb{D}\to\mathbb{S}^2$ is defined as \cite{jost2008}
\begin{align}
E(f):=\frac{1}{2}\int_{\mathbb{D}}\delta^{\alpha\beta}(x)\widetilde{\delta}_{ij}(f(x))\frac{\partial f^i}{\partial x^{\alpha}}\frac{\partial f^j}{\partial x^{\beta}}\sqrt{\det\delta}dx^1 \wedge dx^2\label{DiriE},
\end{align}
where $\left[\delta^{\alpha\beta}\right]:=\left[\delta_{\alpha\beta}\right]^{-1}$ and $\det\delta:=\det\left[\delta_{\alpha\beta}\right]$.

Suppose $\Gamma$ is a fixed smooth map from $\partial\mathbb{D}$ into $\mathbb{S}^2$; the energy has a minimizer of \eqref{DiriE} in the class
\begin{align*}
E_{\Gamma}:=\left\{w\in W^{1,2}\left(\mathbb{D}, \mathbb{R}^3\right) \;\middle\lvert \; \mbox{$w\in\mathbb{S}^2$ a.e. in 
$\mathbb{D}$, $w=\Gamma$ on $\partial\mathbb{D}$}\right\}.
\end{align*}
For each map $w\in E_{\Gamma}$, we can define $W:\mathbb{R}^2\to \mathbb{S}^2$ at a point $x=\left(x^1, x^2\right)$ by
\begin{align*}
\widetilde{w}\left(x\right):=
\left\{\begin{aligned}
&w\left(x\right)\in\mathbb{S}^2_{-}, &\mbox{$x\in \mathbb{D}$},\\
&w\left(x\right)=\underline{w}\left(x\right)\in\Gamma, &\mbox{$x\in \partial\mathbb{D}$},\\
&\underline{w}\left(x\right)\in\mathbb{S}^2_{+}, &\mbox{$x\in \mathbb{R}^2\backslash\mathbb{D}$}
\end{aligned}\right.
\end{align*}
where $\underline{w}\left(x\right):=w\left(\frac{x^1}{\| x\|^2}, \frac{x^2}{\| x\|^2}\right)$ and $\mathbb{S}^2_{\pm}$ denote the upper and lower spherical caps of $\mathbb{S}^2$ bounded by $\Gamma$, respectively.

Next, we use the north pole stereographic projection to map the image from $\mathbb{S}^2$ onto $\overline{\mathbb{C}}$ via $\widetilde{z}:=\Pi_{\mathbb{S}^2}\circ \widetilde{w}$:
\begin{align*}
\widetilde{z}\left(x\right):=
\left\{\begin{aligned}
&z\left(x\right)=u+\mathrm{i} v,& x\in \mathbb{D},\\
&\underline{z}\left(x\right)=\underline{u}+\mathrm{i}\underline{v},& x\in \mathbb{R}^2\backslash\mathbb{D}.
\end{aligned}\right.
\end{align*}
In addition, the north pole inverse stereographic projection at $\left(u, v\right)$ is defined by
\begin{align}
\Pi^{-1}_{\mathbb{S}^2}\left(u, v\right)=\left(\frac{2u}{u^2+v^2+1}, \frac{2v}{u^2+v^2+1}, \frac{u^2+v^2-1}{u^2+v^2+1}\right) \label{eq:inverse_stere_proj}
\end{align}
with the induced metric $g=\frac{4}{\left(u^2+v^2+1\right)^2}\left(du^2+dv^2\right)$.

Since stereographic projections are angle-preserving, we use the north pole stereographic projection to change the target image from a unit sphere $\mathbb{S}^2$ to $\overline{\mathbb{C}}$. We then relate the Dirichlet energy of $\widetilde{w}: (\mathbb{R}^2, \delta)\rightarrow(\mathbb{S}^2, \widetilde{\delta})$ to $\widetilde{z}\equiv\Pi_{\mathbb{S}^2}\circ \widetilde{w}:\left(\mathbb{R}^2, \delta\right)\rightarrow\left(\mathbb{R}^2, g\right)$ as follows:
\begin{subequations}
\begin{align}
E\left(\widetilde{w}\right)&=\frac{1}{2}\int_{\mathbb{R}^2}\delta^{\alpha\beta}\left(x\right)\widetilde{\delta}_{ij}\left(\widetilde{w}\left(x\right)\right)\frac{\partial \widetilde{w}^i}{\partial x^{\alpha}}\frac{\partial \widetilde{w}^j}{\partial x^{\beta}} \sqrt{\det\delta}dx^1\wedge dx^2, \ i, j= 1, 2, 3, \label{DiriE.5} \\ 
E\left(\widetilde{z}\right)&=\frac{1}{2}\int_{\mathbb{R}^2}\delta^{\alpha\beta}\left(x\right)g_{ij}\left(\widetilde{z}\left(x\right)\right)\frac{\partial \widetilde{z}^i}{\partial x^{\alpha}}\frac{\partial \widetilde{z}^j}{\partial x^{\beta}} \sqrt{\det\delta}dx^1\wedge dx^2, \ i, j= 1, 2.\label{DiriE.6}
\end{align}
\end{subequations}
Thus, $\widetilde{w}^*$ is a minimizer of \eqref{DiriE.5} if and only if $\widetilde{z}^*=\Pi_{\mathbb{S}^2}\circ\widetilde{w}^*$ is a minimizer of \eqref{DiriE.6}.

Henceforth, we focus on the study of the Dirichlet energy $E\left(\widetilde{z}\right)$ in \eqref{DiriE.6}. The integral of \eqref{DiriE.6} can be divided into two parts, namely, the southern integral $E_{S}\left(z\right)$ on $\mathbb{D}$ and the northern integral $E_{N}\left(\underline{z}\right)$ on $\mathbb{R}^2\backslash \mathbb{D}$. By applying the inverse stereographic projection \eqref{eq:inverse_stere_proj}, $E(\widetilde{z})$ can be written as the sum of the southern and northern hemispheres:
\begin{subequations}
\begin{align}
\label{Dirichlet:Z}
E\left(\widetilde{z}\right) = E_{S}\left(z\right)+E_{N}\left(\underline{z}\right),
\end{align}
where
\begin{align}
E_{S}(z) 
&=\frac{1}{2}\int_{\mathbb{D}}\delta^{\alpha\beta}g_{ij}\left(z\right)\frac{\partial z^i}{\partial x^{\alpha}}\frac{\partial z^j}{\partial x^{\beta}}\sqrt{\det\delta} dx^1\wedge dx^2 \nonumber \\
&= \int_{\mathbb{D}}\frac{2\langle \nabla z, {\nabla}z\rangle}{\left(1+\lvert  z\rvert^2\right)^2} dx^1\wedge dx^2, \label{DiriE.7} \\
E_{N}(\underline{z}) 
&= \frac{1}{2}\int_{\mathbb{R}^2\backslash\mathbb{D}}\delta^{\alpha\beta}g_{ij}\left(\underline{z}\right)\frac{\partial \underline{z}^i}{\partial x^{\alpha}}\frac{\partial \underline{z}^j}{\partial x^{\beta}}\sqrt{\det\delta}  dx^1\wedge dx^2 \nonumber \\
&= \int_{\mathbb{R}^2\backslash\mathbb{D}}\frac{2\langle \nabla\underline{z}, \nabla\underline{z}\rangle}{\left(1+\lvert  \underline{z}\rvert^2\right)^2} dx^1\wedge dx^2.\label{DiriE.8}
\end{align}
\end{subequations}
Through integration by parts, $E_{S}$ and $E_{N}$ can be expressed as  
\begin{equation}
\begin{split}
E_{S}(z) 
=& -\int_{\mathbb{D}}\left\langle \nabla\left(\frac{2}{(1+\lvert  z\rvert^2)^2}\right), (u{\nabla}u+v{\nabla}v)\right\rangle dx^1\wedge dx^2\\
&\quad\quad-\int_{\mathbb{D}}\frac{2\langle z, \triangle z\rangle}{(1+\lvert  z\rvert^2)^2} dx^1\wedge dx^2 \label{DiriEs.8-2} \\
\end{split}
\end{equation}
\begin{equation}
\begin{split}
E_{N}(\underline{z})
=& -\int_{\mathbb{R}^2\backslash\mathbb{D}}\left\langle \nabla\left(\frac{2}{(1+\lvert  \underline{z}\rvert^2)^2}\right), (\underline{u}\nabla\underline{u}+\underline{v}\nabla\underline{v})\right\rangle dx^1\wedge dx^2\\
&\quad\quad-\int_{\mathbb{R}^2\backslash\mathbb{D}}\frac{2\langle\underline{z}, \triangle \underline{z}\rangle}{(1+\lvert  \underline{z}\rvert^2)^2} dx^1\wedge dx^2, \label{DiriEn.8-2}
\end{split}
\end{equation}
where the boundary integral terms in \eqref{DiriEs.8-2} and \eqref{DiriEn.8-2} are cancelled in \eqref{Dirichlet:Z} because the unit normal vectors along the boundary are in the opposite directions to their domains. 

\subsection{Critical points of $E_{S}(z)$ and $E_{N}(\underline{z})$}
In what follows, we first discuss the critical point for $E_S(z)$, and then $E_N(\underline{z})$ can be handled in a similar way. Since
\begin{align}
&\left\langle \nabla\left(\frac{2}{\left(1+\lvert  z\rvert^2\right)^2}\right), \left(u\nabla u+v\nabla v\right)\right\rangle \nonumber \\
=&-8\frac{\left(u^2u_{x^1}^2+2uvu_{x^1}v_{x^1}+v^2v^2_{x^1}\right)+\left(u^2u_{x^2}^2+2uvu_{x^2}v_{x^2}+v^2v^2_{x^2}\right)}{\left(1+u^2+v^2\right)^3} \nonumber \\
=&-8\frac{\left(\frac{1}{2}\frac{\partial u^2}{\partial x^{1}}+\frac{1}{2}\frac{\partial v^2}{\partial x^{1}}\right)^2+\left(\frac{1}{2}\frac{\partial u^2}{\partial x^{2}}+\frac{1}{2}\frac{\partial v^2}{\partial x^{2}}\right)^2}{\left(1+u^2+v^2\right)^3} \nonumber \\
=&-\frac{2}{1+u^2+v^2}\left\lvert  \frac{\nabla\left(1+u^2+v^2\right)}{1+u^2+v^2}\right\rvert^2 \nonumber \\
=&-\frac{2}{1+\lvert  z\rvert^2}\left\vert\nabla\ln\left(1+\lvert  z\rvert^2\right)\right\rvert^2, \label{ELn.10}
\end{align}
where $u_{x^i} \equiv \frac{\partial u}{\partial x^i}$ and $v_{x^i} \equiv \frac{\partial v}{\partial x^i}$ for $i= 1, 2$, $E_{S}$ in \eqref{DiriEs.8-2} becomes 
\begin{equation} 
E_S(z) 
= \int_{\mathbb{D}}\frac{2}{1+\lvert   z\rvert^2}\left\lvert  \nabla\ln\left(1+\lvert   z\rvert^2\right)\right\rvert^2 dx^1\wedge dx^2 
-\int_{\mathbb{D}}\frac{2\langle z, \triangle z\rangle}{(1+\lvert   z\rvert^2)^2} dx^1\wedge dx^2. \label{DiriEs.8-2-1} 
\end{equation}

Our goal is to find the critical point and the Euler--Lagrange equations for the Dirichlet energy $E(\widetilde{z})$ in \eqref{Dirichlet:Z}, which are the generalization of the geodesics on $\mathbb{S}^2$ and the harmonic equations on $\mathbb{D}$ after applying a stereographic projection to transform the curved target into $\mathbb{C}$. For convenience, we now define the Euler--Lagrange equations for $\mathbf{u}$ and $\mathbf{v}$ as
\begin{subequations} \label{ELs}
\begin{align}
-\triangle \mathbf{u}&=\frac{-2\mathbf{u} \mathbf{u}^2_{x^1}- 2\mathbf{u}\mathbf{u}^2_{x^2}+ 2\mathbf{u}\mathbf{v}^2_{x^1}+ 2\mathbf{u}\mathbf{v}^2_{x^2}- 4\mathbf{u}_{x^1}\mathbf{v}_{x^1}\mathbf{v}- 4\mathbf{u}_{x^2}\mathbf{v}_{x^2}\mathbf{v}}{1+\mathbf{u}^2+\mathbf{v}^2},\\
-\triangle \mathbf{v}&=\frac{-2\mathbf{v}\mathbf{v}^2_{x^1}- 2\mathbf{v}\mathbf{v}^2_{x^2}+ 2\mathbf{v}\mathbf{u}^2_{x^1}+ 2\mathbf{v}\mathbf{u}^2_{x^2}- 4\mathbf{u}_{x^1}\mathbf{v}_{x^1}\mathbf{u}- 4\mathbf{u}_{x^2}\mathbf{v}_{x^2}\mathbf{u}}{1+\mathbf{u}^2+\mathbf{v}^2},  
\end{align}
\end{subequations}
and we give the minimal conditions as
\begin{subequations} \label{assumptions}
\begin{align}
    &\ \lvert  \mathbf{u}\rvert^2\lvert  \nabla \mathbf{v}\rvert^2=\lvert  \mathbf{v}\rvert^2\lvert  \nabla \mathbf{u}\rvert^2, \mbox{ i.e., } \mathbf{u}^2\left(\mathbf{v}^2_{x^1}+\mathbf{v}^2_{x^2}\right) = \mathbf{v}^2\left(\mathbf{u}^2_{x^1}+\mathbf{u}^2_{x^2}\right),\label{assumption:a} \\
    &\ \mbox{$\nabla \mathbf{u}$ and $\nabla \mathbf{v}$ are linearly dependent.}\label{assumption:b}
\end{align}
\end{subequations}
\begin{theorem} \label{lem:CT_Es}
Suppose $z_c=u+\mathrm{i}v$ is a critical point of $E_S$ in \eqref{DiriE.7}. Then, $u$ and $v$ satisfy the Euler--Lagrange equations in \eqref{ELs} with $(u, v) =  (\mathbf{u}, \mathbf{v})$ 
and
\begin{align} 
E_S(z_c)
=& \int_{\mathbb{D}}\frac{1}{1+\lvert  z_c\rvert^2}\left\lvert  \nabla\ln\left(1+\lvert  z_c\rvert^2\right)\right\rvert^2 dx^1\wedge dx^2 \nonumber \\
&\quad +\int_{\mathbb{D}}\frac{4}{1+\lvert  z_c\rvert^2}\frac{u^2\lvert  \nabla v\rvert^2-2uv\langle \nabla u, \nabla v\rangle +v^2\lvert  \nabla u\rvert^2}{\left(1+\lvert  z_c\rvert^2\right)^2} dx^1\wedge dx^2. \label{eq:Es_zc-1} 
\end{align}
Furthermore, if $u$ and $v$ also satisfy the minimal conditions \eqref{assumptions} almost everywhere, then $E_S$ reaches its minimum at $z_c$ and
\begin{equation}
E_S(z_c)=\int_{\mathbb{D}}\frac{1}{1+\lvert  z_c\rvert^2}\left\lvert  \nabla \ln\left(1+\lvert  z_c\rvert^2\right)\right\rvert^2 dx^1\wedge dx^2. \label{Es_zn.6}
\end{equation}
\end{theorem}
\begin{proof}
Let $\phi\equiv\phi^1+\mathrm{i}\phi^2\in C^{\infty}_c(\mathbb{D}, \mathbb{C})$ be any locally smooth map with compact support. Consider the perturbation $z_c+t\phi$ with a parameter $t\in\mathbb{R}$. From \eqref{DiriE.7}, we have, in local coordinates for all $\lvert t\lvert \to 0$,
\begin{align}
0=&\left. \frac{d}{dt}E_{S}\left(z_c+t\phi\right) \right\rvert_{t=0}\nonumber \\
=&\left. \frac{d}{dt}\int_{\mathbb{D}}\frac{2}{\left(1+\lvert z_c+t\phi\rvert^2\right)^2}\langle \nabla(z_c+t\phi), \nabla(z_c+t\phi) \rangle dx^1\wedge dx^2 \right\rvert_{t=0}\nonumber\\
=&\sum_{i=1}^2\left.\frac{d}{dt}\int_{\mathbb{D}}\frac{2}{\left(1+(u+\phi^1t)^2+(v+\phi^2t)^2\right)^2}\left(\frac{\partial (u+\phi^1t)}{\partial x^i}\right)^2 dx^1\wedge dx^2\right\rvert_{t=0}\nonumber\\
&\quad+\sum_{i=1}^2\left.\frac{d}{dt}\int_{\mathbb{D}}\frac{2}{\left(1+(u+\phi^1t)^2+(v+\phi^2t)^2\right)^2}\left(\frac{\partial (v+\phi^2t)}{\partial x^i}\right)^2dx^1\wedge dx^2\right\rvert_{t=0}\nonumber\\
=&\sum_{i=1}^2\int_{\mathbb{D}}\frac{-8(u\phi^1+v\phi^2)}{\left(1+u^2+v^2\right)^3}\left[\left(\frac{\partial u}{\partial x^i}\right)^2+\left(\frac{\partial v}{\partial x^i}\right)^2\right]dx^1\wedge dx^2\nonumber\\
&\quad+\sum_{i=1}^2\int_{\mathbb{D}}\frac{4}{\left(1+u^2+v^2\right)^2}\left(\frac{\partial u}{\partial x^i}\frac{\partial \phi^1}{\partial x^i}+\frac{\partial v}{\partial x^i}\frac{\partial \phi^2}{\partial x^i}\right)dx^1\wedge dx^2\nonumber\\
=&\int_{\mathbb{D}}\frac{-8\langle \nabla z, \nabla z\rangle}{\left(1+\lvert z\rvert^2\right)^3}\langle z, \phi\rangle dx^1\wedge dx^2+\int_{\mathbb{D}}\frac{4}{\left(1+\lvert z\rvert^2\right)^2}\langle \nabla z, \nabla\phi\rangle dx^1\wedge dx^2\nonumber\\
:=&\textup{I}+\textup{II}. \label{ELs.1}
\end{align}
Then,
\begin{subequations}
\begin{equation}
\textup{I}=\int_{\mathbb{D}}\frac{-8}{\left(1+u^2+v^2\right)^3}\left[u^2_{x^1}+u^2_{x^2}+v^2_{x^1}+v^2_{x^2}\right](u\phi^1+v\phi^2) dx^1\wedge dx^2,\label{ELs.2}
\end{equation}
and by integration by parts, we have
\begin{align}
\textup{II}=&\int_{\mathbb{D}}\left(16\frac{uu_{x^1}+vv_{x^1}}{\left(1+u^2+v^2\right)^3}u_{x^1}-\frac{4}{\left(1+u^2+v^2\right)^2}u_{x^1x^1}\right)\phi^1dx^1\wedge dx^2\nonumber\\
&+\int_{\mathbb{D}}\left(16\frac{uu_{x^1}+vv_{x^1}}{\left(1+u^2+v^2\right)^3}v_{x^1}-\frac{4}{\left(1+u^2+v^2\right)^2}v_{x^1x^1}\right)\phi^2dx^1\wedge dx^2\nonumber\\
&+\int_{\mathbb{D}}\left(16\frac{uu_{x^2}+vv_{x^2}}{\left(1+u^2+v^2\right)^3}u_{x^2}-\frac{4}{\left(1+u^2+v^2\right)^2}u_{x^2x^2}\right)\phi^1dx^1\wedge dx^2\nonumber\\
& +\int_{\mathbb{D}}\left(16\frac{uu_{x^2}+vv_{x^2}}{\left(1+u^2+v^2\right)^3}v_{x^2}-\frac{4}{\left(1+u^2+v^2\right)^2}v_{x^2x^2}\right)\phi^2dx^1\wedge dx^2\nonumber\\
=  
&\int_{\mathbb{D}}\frac{4}{\left(1+\lvert z\rvert^2\right)^2}\left(\frac{4u\lvert \nabla u\rvert^2+4v\langle \nabla v, \nabla u\rangle}{1+\lvert z\rvert^2}-\triangle u\right)\phi^1dx^1\wedge dx^2 \nonumber\\
&+\int_{\mathbb{D}}\frac{4}{\left(1+\lvert z\rvert^2\right)^2}\left(\frac{4v\lvert \nabla v\rvert^2+4u\langle \nabla u, \nabla v\rangle}{1+\lvert z\rvert^2}-\triangle v\right)\phi^2dx^1\wedge dx^2.\label{ELs.3} 
\end{align}
\end{subequations}
Substituting \eqref{ELs.2} and \eqref{ELs.3} into \eqref{ELs.1}, we have
\begin{align*}
0=&\left. \frac{d}{dt}E_{S}\left(z_c+t\phi\right) \right\rvert_{t=0}\\
=&\int_{\mathbb{D}}\frac{4}{(1+\lvert z\rvert^2)^2}\left[\frac{-2\left(\lvert \nabla u\rvert^2+\lvert \nabla v\rvert^2\right)u}{1+\lvert z\rvert^2}+\frac{4u\lvert \nabla u\rvert^2+4v\langle \nabla v, \nabla u\rangle}{1+\lvert z\rvert^2}-\triangle u \right]\phi^1 dx^1\wedge dx^2\\
&+\int_{\mathbb{D}}\frac{4}{\left(1+\lvert z\rvert^2\right)^2}\left[\frac{-2(\lvert \nabla u\rvert^2+\lvert \nabla v\rvert^2)v}{1+\lvert z\rvert^2}+\frac{4v\lvert \nabla v\rvert^2+4u\langle \nabla u, \nabla v\rangle}{1+\lvert z\rvert^2}-\triangle v\right]\phi^2 dx^1\wedge dx^2.
\end{align*}
That is, $u$ and $v$ must satisfy the Euler--Lagrange equations in \eqref{ELs}.

Substituting the results of \eqref{ELs} into the second term of \eqref{DiriEs.8-2}, we have that
\begin{align}
\frac{\langle z_c, -\triangle z_c\rangle}{1+\lvert z_c\rvert^2} 
=&\frac{\left(-2u^2u^2_{x^1}-4uvu_{x^1}v_{x^1}-2v^2v^2_{x^1}\right)+\left(-2u^2u^2_{x^2}-4uvu_{x^2}v_{x^2}-2v^2v^2_{x^2}\right)}{\left(1+u^2+v^2\right)^2} \nonumber \\
&+\frac{2u^2v^2_{x^1}+2u^2v^2_{x^2}+2v^2u^2_{x^1}+2v^2u^2_{x^2}-4uvu_{x^1}v_{x^1}-4uvu_{x^2}v_{x^2}}{\left(1+u^2+v^2\right)^2} \nonumber \\
=&-2\frac{\left(uu_{x^1}+vv_{x^1}\right)^2+\left(uu_{x^2}+vv_{x^2}\right)^2}{\left(1+u^2+v^2\right)^2} \nonumber \\
&+2\frac{u^2\left(v^2_{x^1}+v^2_{x^2}\right)-2uv\left(u_{x^1}v_{x^1}+u_{x^2}v_{x^2}\right)+v^2\left(u^2_{x^1}+u^2_{x^2}\right)}{\left(1+u^2+v^2\right)^2} \nonumber \\
=&-\frac{1}{2}\left\rvert\nabla\ln\left(1+\lvert z_c\rvert^2\right)\right\rvert^2+2\frac{u^2\lvert \nabla v\rvert^2-2uv\langle \nabla u, \nabla v\rangle +v^2\lvert \nabla u\rvert^2}{\left(1+\lvert z_c\rvert^2\right)^2}. \label{ELn.11} 
\end{align}
Plugging \eqref{ELn.11} into \eqref{DiriEs.8-2-1}, $E_S(z_c)$ can be rewritten in the form of \eqref{eq:Es_zc-1}.

By applying the arithmetic--geometric mean inequality and the Cauchy inequality 
\begin{align*}
\frac{u^2\lvert \nabla v\rvert^2 +v^2\lvert \nabla u\rvert^2}{2}\geq \sqrt{u^2v^2\lvert \nabla v\rvert^2\lvert \nabla u\rvert^2},\ \ \ 
   \lvert \nabla u\rvert^2  \lvert \nabla v\rvert^2 \geq \lvert \langle \nabla u, \nabla v\rangle\rvert^2
\end{align*}
to \eqref{eq:Es_zc-1}, we show that
\begin{align*}
E_S(z_c) &\geq \int_{\mathbb{D}}\frac{1}{1+\lvert z_c\rvert^2}\left\rvert\nabla\ln\left(1+\lvert z_c\rvert^2\right)\right\rvert^2 dx^1\wedge dx^2\\
&\quad +\int_{\mathbb{D}}\frac{4}{1+\lvert z_c\rvert^2}\frac{2\sqrt{u^2v^2\lvert \nabla v\rvert^2\lvert \nabla u\rvert^2}-2uv\langle \nabla u, \nabla v\rangle}{\left(1+\lvert z_c\rvert^2\right)^2} dx^1\wedge dx^2 \\ 
& \geq \int_{\mathbb{D}}\frac{1}{1+\lvert z_c\rvert^2}\left\rvert\nabla\ln\left(1+\lvert z_c\rvert^2\right)\right\rvert^2 dx^1\wedge dx^2,  
\end{align*}
which implies that $E_S$ reaches its minimum at $z_c=u+\mathrm{i}v$ if the minimal conditions \eqref{assumptions} hold.
\end{proof}

Similarly, we consider the variation of $E_N(\underline{z})$ in \eqref{DiriE.8}. We assume that $\underline{z}_c=\underline{u}+\mathrm{i}\underline{v}$ is a smooth critical point of $E_N$ at infinity. To make sense of the perturbations at infinity, we first use the inversion $\zeta \equiv \frac{1}{\bar{\underline{z}}_c}=\mu+\mathrm{i}\nu$ to ensure the finiteness of the points around infinity and then rewrite the energy $E_N$ and its perturbation in $\zeta$.

\begin{theorem} \label{lem:CT_En}
Suppose $\underline{z}_c=\underline{u}+\mathrm{i}\underline{v}$ is a critical point of $E_N$ in \eqref{DiriE.8}, and define $\zeta=\frac{1}{\bar{\underline{z}}_c}=\mu+\mathrm{i}\nu$. Then, $\mu$ and $\nu$ satisfy the Euler--Lagrange equations in \eqref{ELs} with $(\mu, \nu) = (\mathbf{u}, \mathbf{v})$, 
and
\begin{align}
E_N(\zeta)&=\int_{\mathbb{R}^2\backslash\mathbb{D}}\frac{-\left\rvert\nabla \ln\left(1+\lvert \zeta\rvert^2\right)\right\rvert^2}{1+\lvert \zeta\rvert^2} dx^1\wedge dx^2 
 +\int_{\mathbb{R}^2\backslash\mathbb{D}}\frac{4\langle \nabla\zeta, \nabla\zeta\rangle}{\left(1+\lvert \zeta\rvert^2\right)^2} dx^1\wedge dx^2 \nonumber \\
&\quad -\int_{\mathbb{R}^2\backslash\mathbb{D}}\frac{4}{1+\lvert \zeta\rvert^2}\frac{\mu^2\lvert \nabla\nu\rvert^2-2\mu\nu\langle \nabla\mu, \nabla\nu\rangle +\nu^2\lvert \nabla\mu\rvert^2}{\left(1+\lvert \zeta\rvert^2\right)^2} dx^1\wedge dx^2. \label{ELn.7}
\end{align}
Moreover, if $\mu$ and $\nu$ also satisfy the minimal conditions \eqref{assumptions} almost everywhere, then, $E_N$ reaches its minimum at $\zeta$ and
\begin{align}
E_N(\zeta)=&-\int_{\mathbb{R}^2\backslash\mathbb{D}}\frac{\left\rvert\nabla\ln\left(1+\lvert \zeta\rvert^2\right)\right\rvert^2}{1+\lvert \zeta\rvert^2} dx^1\wedge dx^2 
 +\int_{\mathbb{R}^2\backslash\mathbb{D}}\frac{4\langle \nabla\zeta, \nabla\zeta\rangle}{\left(1+\lvert \zeta\rvert^2\right)^2} dx^1\wedge dx^2. \label{ELn.8}
\end{align}
\end{theorem}
\begin{proof}
With $\zeta=\frac{1}{\bar{\underline{z}}_c}=\mu+\mathrm{i}\nu$ and 
\begin{align*}
\left\langle \nabla\frac{1}{\bar{\underline{z}}_c}, \nabla\frac{1}{\bar{\underline{z}}_c}\right\rangle=&\sum_{i=1}^2\left( \left(\frac{\partial}{\partial x^i}\frac{\underline{u}}{\underline{u}^2+\underline{v}^2}\right)^2+\left(\frac{\partial}{\partial x^i}\frac{\underline{v}}{\underline{u}^2+\underline{v}^2}\right)^2\right)\nonumber\\
=& \frac{\underline{u}^2_{x^1}+\underline{v}^2_{x^1}}{\left(\underline{u}^2+\underline{v}^2\right)^2}+\frac{\underline{u}^2_{x^2}+\underline{v}^2_{x^2}}{\left(\underline{u}^2+\underline{v}^2\right)^2} 
=\frac{1}{\lvert \underline{z}_c\rvert^4}\langle \nabla\underline{z}_c, \nabla\underline{z}_c\rangle,
\end{align*}
$E_{N}(\underline{z}_c)$ in \eqref{DiriE.8} can be rewritten as
\begin{align*}
E_{N}(\underline{z}_c) &= \int_{\mathbb{R}^2\backslash\mathbb{D}}\frac{2}{\left(1+\frac{1}{\lvert \underline{z}_c\rvert^2}\right)^2}\langle \nabla\frac{1}{\bar{\underline{z}}_c}, \nabla\frac{1}{\bar{\underline{z}}_c}\rangle dx^1\wedge dx^2\\
&= \int_{\mathbb{R}^2\backslash\mathbb{D}}\frac{2\langle \nabla\zeta, \nabla\zeta\rangle}{\left(1+\lvert \zeta\rvert^2\right)^2} dx^1\wedge dx^2\equiv E_{N}(\zeta). 
\end{align*}
By similar arguments to those for \eqref{ELs.1} in Theorem~\ref{lem:CT_Es}, we obtain the Euler--Lagrange equations in \eqref{ELs} for $\mu$ and $\nu$.

By using the fact that $\langle q, \frac{1}{\overline{q}}\rangle=1$ with $q\in\mathbb{C}$, we have the following useful relation for the Laplacian operator:
\begin{align}
&0=\langle \triangle q, \frac{1}{\overline{q}}\rangle+\langle q, \triangle\frac{1}{\overline{q}}\rangle+2\langle \nabla q, \nabla\frac{1}{\overline{q}}\rangle\nonumber\\
\Rightarrow&-\langle q, \triangle q\rangle=\lvert q\rvert^4\langle \frac{1}{\overline{q}}, \triangle\frac{1}{\overline{q}}\rangle+2\lvert q\rvert^2\langle \nabla q, \nabla\frac{1}{\overline{q}}\rangle. \label{iden}
\end{align}
Via \eqref{iden}, $E_N(\underline{z}_c)$ in \eqref{DiriEn.8-2} can be simplified and rewritten as $E_N(\zeta)$ with $\zeta =\frac{1}{\bar{\underline{z}}_c}$:
\begin{align}
E_N(\underline{z}_c) &= -\int_{\mathbb{R}^2\backslash\mathbb{D}}\left\langle \nabla\left(\frac{2}{\left(1+\lvert \underline{z}_c\rvert^2\right)^2}\right), \left(\underline{u}\nabla\underline{u}+\underline{v}\nabla\underline{v}\right)\right\rangle dx^1\wedge dx^2 \nonumber \\
&\quad +\int_{\mathbb{R}^2\backslash\mathbb{D}}\frac{2}{\left(1+\lvert \underline{z}_c\rvert^2\right)^2}\left(\lvert \underline{z}_c\rvert^4\langle \frac{1}{\bar{\underline{z}}_c}, \triangle\frac{1}{\bar{\underline{z}}_c}\rangle+ 2\lvert \underline{z}_c\rvert^2\langle \nabla\underline{z}_c, \nabla\frac{1}{\bar{\underline{z}}_c}\rangle\right) dx^1\wedge dx^2 \nonumber  \\
&= \int_{\mathbb{R}^2\backslash\mathbb{D}}\frac{4\langle \nabla\underline{z}_c, \nabla\underline{z}_c\rangle}{\left(1+\lvert \underline{z}_c\rvert^2\right)^2} dx^1\wedge dx^2 
 -\int_{\mathbb{R}^2\backslash\mathbb{D}}\frac{2\lvert \underline{z}_c\rvert^2}{1+\lvert \underline{z}_c\rvert^2}\left\rvert \nabla\ln\left(\frac{1+\lvert \underline{z}_c\rvert^2}{\lvert \underline{z}_c\rvert^2}\right)\right\rvert^2 dx^1\wedge dx^2 \nonumber \\
&\quad +\int_{\mathbb{R}^2\backslash\mathbb{D}}\frac{2}{\left(1+\lvert \underline{z}_c\rvert^2\right)^2}\left(\lvert \underline{z}_c\rvert^4\langle \frac{1}{\bar{\underline{z}}_c}, \triangle\frac{1}{\bar{\underline{z}}_c}\rangle\right) dx^1\wedge dx^2 \nonumber \\
&= \int_{\mathbb{R}^2\backslash\mathbb{D}}\left(\frac{4\lvert \underline{z}_c\rvert^4}{\left(1+\lvert \underline{z}_c\rvert^2\right)^2}\langle \nabla\frac{1}{\underline{z}_c}, \nabla\frac{1}{\underline{z}_c}\rangle -  \frac{2\lvert \underline{z}_c\rvert^2}{1+\lvert \underline{z}_c\rvert^2}\left\rvert \nabla\ln\left(\frac{1+\lvert \underline{z}_c\rvert^2}{\lvert \underline{z}_c\rvert^2}\right)\right\rvert^2 \right) dx^1\wedge dx^2 \nonumber \\
&\quad +\int_{\mathbb{R}^2\backslash\mathbb{D}}\frac{2}{\left(1+\lvert \underline{z}_c\rvert^2\right)^2}\left(\lvert \underline{z}_c\rvert^4\langle \frac{1}{\bar{\underline{z}}_c}, \triangle\frac{1}{\bar{\underline{z}}_c}\rangle\right) dx^1\wedge dx^2 \nonumber \\
&= \int_{\mathbb{R}^2\backslash\mathbb{D}} \left( \frac{4\langle \nabla\zeta, \nabla\zeta\rangle}{\left(1+\lvert \zeta\rvert^2\right)^2} - \frac{2}{1+\lvert \zeta\rvert^2}\left\rvert \nabla\ln\left(1+\lvert \zeta\rvert^2\right)\right\rvert^2 + \frac{2\langle \zeta, \triangle\zeta\rangle}{\left(1+\lvert \zeta\rvert^2\right)^2} \right) dx^1\wedge dx^2 \nonumber \\ 
&\equiv E_N(\zeta).  \label{identity.1} 
\end{align}
With the same arguments as for \eqref{ELn.11}, we have
\begin{align}
\frac{\langle \zeta, \triangle \zeta\rangle}{1+\lvert \zeta\rvert^2}=&\frac{1}{2}\left\rvert\nabla \ln\left(1+\lvert \zeta\rvert^2\right)\right\rvert^2-2\frac{\nu^2\lvert \nabla\mu\rvert^2-2\nu\mu\langle \nabla\nu, \nabla\mu\rangle +\mu^2\lvert \nabla\nu\rvert^2}{\left(1+\lvert \zeta\rvert^2\right)^2}. \label{identity.2}
\end{align}
Plugging \eqref{identity.2} into the energy $E_N(\zeta)$ of \eqref{identity.1}, if $\mu$ and $\nu$ both satisfy the minimal conditions \eqref{assumptions}, then $E_N(\zeta)$ reaches its minimum
\begin{align*}
E_N(\zeta) &= \int_{\mathbb{R}^2\backslash\mathbb{D}}\frac{4\langle \nabla\zeta, \nabla\zeta\rangle}{(1+\lvert \zeta\rvert^2)^2} dx^1\wedge dx^2 
  -\int_{\mathbb{R}^2\backslash\mathbb{D}}\frac{\left\rvert\nabla\ln(1+\lvert \zeta\rvert^2)\right\rvert^2}{1+\lvert \zeta\rvert^2} dx^1\wedge dx^2\\
&\quad -\int_{\mathbb{R}^2\backslash\mathbb{D}}\frac{4}{1+\lvert \zeta\rvert^2}\frac{\mu^2\lvert \nabla\nu\rvert^2-2\mu\nu\langle \nabla\mu, \nabla\nu\rangle +\nu^2\lvert \nabla\mu\rvert^2}{(1+\lvert \zeta\rvert^2)^2} dx^1\wedge dx^2.\\
&\geq \int_{\mathbb{R}^2\backslash\mathbb{D}}\frac{4 \langle \nabla\zeta, \nabla\zeta\rangle}{(1+\lvert \zeta\rvert^2)^2} dx^1\wedge dx^2 
 -\int_{\mathbb{R}^2\backslash\mathbb{D}}\frac{\left\rvert\nabla\ln(1+\lvert \zeta\rvert^2)\right\rvert^2}{1+\lvert \zeta\rvert^2} dx^1\wedge dx^2. 
\end{align*}
\end{proof}

\subsection{Homothetic and minimal energies}
Since the conclusions of $E_S(z)$ in Theorem~\ref{lem:CT_Es} and $E_N(\underline{z})$ in Theorem~\ref{lem:CT_En} are similar, we will focus on the southern integral $E_S(z)$ in what follows. Assuming that $z$ is a solution of the Euler--Lagrange equations satisfying the minimal conditions \eqref{assumption:a} and \eqref{assumption:b}, we know that it determines the minimal Dirichlet energy. Furthermore, for closed surfaces of genus zero, the uniformization theorem \cite{ahlfors2010} also guarantees that the map minimizing the Dirichlet energy is a conformal map. Conversely, a map that reduces the Dirichlet energy does not necessarily improve conformality compared to the previous map. For example, the shrinking homothetic transformation does not enhance conformality, but it does reduce the Dirichlet energy. To compare conformality through the calculus of variations, we must study the Euler--Lagrange equations and assumptions \eqref{assumptions} more deeply.

For a positive $\rho\in\mathbb{R}_{+}$, we now consider a new north pole stereographic projection
\begin{align}
\label{Stereo}
\Pi_{\mathbb{S}^2_{\rho}}: (\mathbb{S}^2_{\rho}, \widetilde{\delta}) &\longrightarrow (\mathbb{R}^2, g_{\rho})\\
(f^1, f^2, f^3) & \longmapsto(u, v),\notag
\end{align}
where $g_{\rho}=\frac{4\rho^2}{\left(u^2+v^2+1\right)^2}\left(du^2+dv^2\right)$ is the induced metric.
The inverse of \eqref{Stereo} is given by
\begin{align*}
\Pi^{-1}_{\mathbb{S}^2_{\rho}}\left(\left(u, v\right)\right)=\rho\left(\Pi_{\mathbb{S}^2}^{-1}(u, v)\right).
\end{align*}
Then, the southern integral $E_S(z)$ with respect to $g_{\rho}$ in \eqref{DiriEs.8-2-1} becomes
\begin{align}
E_S(z; g_{\rho})=&\rho^2E_{S}\left(z; g\right) \nonumber \\ 
=&\int_{\mathbb{D}}\frac{2\rho^2}{1+\lvert z\rvert^2}\left\lvert \nabla\ln\left(1+\lvert z\rvert^2\right)\right\rvert^2 dx^1\wedge dx^2
-\int_{\mathbb{D}}\frac{2\rho^2\langle z, \triangle z\rangle}{(1+\lvert z\rvert^2)^2} dx^1\wedge dx^2 \nonumber \\ 
:=& \textup{I}+\textup{II}, \label{DiriE:rho}
\end{align}
which means that $E_S(z; g_{\rho})$ has the same Euler--Lagrange equations and minimum conditions as in \eqref{ELs} and \eqref{assumptions}.

The Dirichlet energy \eqref{DiriE:rho} is divided into an initial state $\textup{I}$ of \eqref{DiriE:rho} determined only by the induced metric $g_{\rho}$, as well as a residual term $\textup{II}$ of \eqref{DiriE:rho} controlled by the Euler--Lagrange equations \eqref{ELs} and the minimal conditions \eqref{assumptions}.
Next, we divide the Euler--Lagrange equations in \eqref{ELs} into two parts:
\begin{subequations}
\begin{align}
-\triangle u=&\underbrace{\frac{-2uu^2_{x^1}-2uu^2_{x^2}-2u_{x^1}v_{x^1}v-2u_{x^2}v_{x^2}v}{1+u^2+v^2}}_{\text{Homothetic}} \label{eq:Homothetic_u}\\
&\quad\quad+\underbrace{\frac{2uv^2_{x^1}+2uv^2_{x^2}-2u_{x^1}v_{x^1}v-2u_{x^2}v_{x^2}v}{1+u^2+v^2}}_{\text{Minimal}}, \label{eq:Minimal_u}\\
-\triangle v=&\underbrace{\frac{-2vv^2_{x^1}-2vv^2_{x^2}-2u_{x^1}v_{x^1}u-2u_{x^2}v_{x^2}u}{1+u^2+v^2}}_{\text{Homothetic}} \label{eq:Homothetic_v} \\
&\quad\quad+\underbrace{\frac{2vu^2_{x^1}+2vu^2_{x^2}-2u_{x^1}v_{x^1}u-2u_{x^2}v_{x^2}u}{1+u^2+v^2}}_{\text{Minimal}}. \label{eq:Minimal_v}
\end{align}
\end{subequations}
Plugging the homothetic terms in \eqref{eq:Homothetic_u} and \eqref{eq:Homothetic_v} and the minimal terms in \eqref{eq:Minimal_u} and \eqref{eq:Minimal_v} into $\textup{II}$ of \eqref{DiriE:rho}, we have 
\begin{align}
\textup{II} &= \int_{\mathbb{D}}\frac{2\rho^2\langle z, -\triangle z\rangle}{(1+\lvert z\rvert^2)^2} dx^1\wedge dx^2 \nonumber\\
&=\underbrace{-\int_{\mathbb{D}}\frac{\rho^2}{1+\lvert z\rvert^2}\left\lvert \nabla\ln\left(1+\lvert z\rvert^2\right)\right\rvert^2 dx^1\wedge dx^2}_{\text{Homothetic}}\nonumber \\
&\quad+\underbrace{\int_{\mathbb{D}}\frac{4\rho^2}{1+\lvert z\rvert^2}\frac{u^2\lvert \nabla v\rvert^2-2uv\langle \nabla u, \nabla v\rangle +v^2\lvert \nabla u\rvert^2}{(1+\lvert z\rvert^2)^2} dx^1\wedge dx^2}_{\text{Minimal}}. \label{integral}
\end{align}
It is clear that the minimal term of \eqref{integral} is nonnegative by the arithmetic--geometric mean and the Cauchy inequalities. Moreover, it reaches its minimum value of zero when the minimal conditions \eqref{assumptions} hold. On the other hand, the homothetic term of \eqref{integral} derived from the homothetic terms in \eqref{eq:Homothetic_u} and \eqref{eq:Homothetic_v} is always reduced by half of the initial state $\textup{I}$ of \eqref{DiriE:rho}. In summary, the Dirichlet energy \eqref{DiriE:rho} becomes
\begin{align} 
E_S(z; g_{\rho})&=\underbrace{\int_{\mathbb{D}}\frac{\rho^2}{1+\lvert z\rvert^2}\left\lvert \nabla\ln\left(1+\lvert z\rvert^2\right)\right\rvert^2 dx^1\wedge dx^2}_{\text{Homothetic}} \nonumber \\
&\quad+\underbrace{\int_{\mathbb{D}}\frac{4\rho^2}{1+\lvert z\rvert^2}\frac{u^2\lvert \nabla v\rvert^2-2uv\langle \nabla u, \nabla v\rangle +v^2\lvert \nabla u\rvert^2}{(1+\lvert z\rvert^2)^2} dx^1\wedge dx^2}_{\text{Minimal}} \label{ESr} \\
&\geq\int_{\mathbb{D}}\frac{\rho^2}{1+\lvert z\rvert^2}\left\lvert \nabla\ln\left(1+\lvert z\rvert^2\right)\right\rvert^2 dx^1\wedge dx^2. \nonumber 
\end{align}

\subsection{Equivalence class and Laplacian equations}
By the Euler--Lagrange equations \eqref{ELs} and the minimal conditions \eqref{assumptions}, the energy in \eqref{ESr} must be decreasing compared to that of the initial state $\textup{I}$ of \eqref{DiriE:rho}. However, the reduction of the first term of \eqref{integral} does not contribute to the conformality because the decrease in the integral can be modified back to the initial state $\textup{I}$ by a homothetic transformation. For example, we can replace $\rho$ with $\sqrt{2l}\hat{\rho}$, $l\in\mathbb{R}_{+}$, in \eqref{ESr}:
\begin{align}
E_S(z; g_{\sqrt{2l}\hat{\rho}})
&= 
 l\int_{\mathbb{D}}\frac{8\hat{\rho}^2}{1+\lvert z\rvert^2}\frac{u^2\lvert \nabla v\rvert^2-2uv\langle \nabla u, \nabla v\rangle +v^2\lvert \nabla u\rvert^2}{(1+\lvert z\rvert^2)^2} dx^1\wedge dx^2 \nonumber \\
& \quad + l \int_{\mathbb{D}}\frac{2\hat{\rho}^2}{1+\lvert z\rvert^2}\left\lvert \nabla\ln\left(1+\lvert z\rvert^2\right)\right\rvert^2 dx^1\wedge dx^2 \label{ES2lr} \\ 
&\geq l\int_{\mathbb{D}}\frac{2\hat{\rho}^2}{1+\lvert z\rvert^2}\left\lvert \nabla\ln\left(1+\lvert z\rvert^2\right)\right\rvert^2 dx^1\wedge dx^2.\nonumber
\end{align}
In other words, we deform the initial metric $g_{\rho}$ conformally into another metric $g_{\sqrt{2l}\hat{\rho}}$ on $\mathbb{R}^2$. Because $g_{\rho}$ and $g_{\sqrt{2l}\hat{\rho}}$ are conformal to each other and the deformation acts on the scaling of the ambient space $\mathbb{R}^2$ rather than the angular structure of the image, it does not improve the conformality of the mappings. To put it another way, it is meaningless to compare the change in the quantity $\int_{\mathbb{D}}\frac{\rho^2}{1+\lvert z\rvert^2}\left\lvert \nabla\ln\left(1+\lvert z\rvert^2\right)\right\rvert^2 dx^1\wedge dx^2$ because we can choose a sufficiently large $l$ so that the minimizer \eqref{ES2lr} is larger than any positive constant. In contrast, the first term of \eqref{ES2lr} is always zero under the minimum condition regardless of the change in $l$.

To compare the amount of change that is effective in improving angle retention, we consider the following admissible set $\mathcal{X}$, which achieves the minimal Dirichlet energy for any initial metric $g_{\rho}$:
\begin{equation*}
\mathcal{X}:=\left\lbrace (z, g_{\rho}) \;\middle\lvert \;
\begin{tabular}{@{}l@{}}
    $z$ satisfies \eqref{ELs} and minimal conditions \eqref{assumptions},\\
    $g_{\rho}$ is the metric induced on $\mathbb{C}$ by $\Pi_{\mathbb{S}^2_{\rho}}:\mathbb{S}^2_{\rho}\to\mathbb{C}$ 
\end{tabular}
\right\rbrace
\end{equation*}
We define the equivalence relation $\sim$ on $\mathcal{X}$ as follows: Given two elements $(z_1, g_{\rho_1})$ and $(z_2, g_{\rho_2})$ in $\mathcal{X}$, we say that
\begin{align}
\label{sim}
(z_1, g_{\rho_1})\sim(z_2, g_{\rho_2})
\end{align}
if and only if there is a $c:=\rho_1^2/\rho_2^2 \in\mathbb{R}_{+}$
such that
\begin{equation*}
    E_S(z_1; g_{\rho_1})=cE_S(z_2; g_{\rho_2}).
\end{equation*}
Hence, we define the Dirichlet energy of the equivalence class $\left(\mathcal{X}, \sim\right)$ as 
\begin{align} 
E_S(\mathcal{X})&=\left\{\int_{\mathbb{D}}\frac{c}{1+\lvert z\rvert^2}\left\lvert \nabla\ln\left(1+\lvert z\rvert^2\right)\right\rvert^2 dx^1\wedge dx^2\;\middle\lvert \; \forall c\in\mathbb{R}_{+}\right\} \nonumber \\
&:=\left[\int_{\mathbb{D}}\frac{1}{1+\lvert z\rvert^2}\left\lvert \nabla\ln\left(1+\lvert z\rvert^2\right)\right\rvert^2 dx^1\wedge dx^2\right]. \label{equiv:energy} 
\end{align}
Similarly, we can define the corresponding concepts $\underline{\mathcal{X}}$ and $E_N(\underline{\mathcal{X}})$ for the northern integral, and we omit the detailed formulation here. Under the equivalence relations \eqref{sim} and \eqref{equiv:energy} of Dirichlet energy, we have the statements in Theorem~\ref{eq:alternative} below for the southern and northern integrals. These equivalent equations have the same effectiveness in improving the conformality of the mapping in the minimization of the Dirichlet energy under the minimal conditions \eqref{assumptions}.
\begin{theorem}
\label{eq:alternative}
Minimizing $E_S$ of \eqref{DiriEs.8-2} and $E_N$ of \eqref{DiriEn.8-2} by Euler--Lagrange equations \eqref{ELs} with minimal conditions \eqref{assumptions} is equivalent to solving
\begin{subequations} \label{eq:homothetic_equivalence}
\begin{align}   
-\triangle \mathbf{u}&=\frac{2\mathbf{u}\mathbf{v}^2_{x^1} +2\mathbf{u}\mathbf{v}^2_{x^2} -2\mathbf{u}_{x^1}\mathbf{v}_{x^1}\mathbf{v} -2\mathbf{u}_{x^2}\mathbf{v}_{x^2}\mathbf{v}}{1+\mathbf{u}^2+\mathbf{v}^2}, \label{eq:homothetic_equivalence_u}\\ 
-\triangle \mathbf{v}&=\frac{2\mathbf{v}\mathbf{u}^2_{x^1} +2\mathbf{v}\mathbf{u}^2_{x^2} -2\mathbf{u}_{x^1}\mathbf{v}_{x^1}\mathbf{u} -2\mathbf{u}_{x^2}\mathbf{v}_{x^2}\mathbf{u}}{1+\mathbf{u}^2+\mathbf{v}^2} \label{eq:homothetic_equivalence_v}
\end{align}
\end{subequations}
with the same minimal conditions \eqref{assumptions}.
Regardless of whether the Euler--Lagrange equations \eqref{ELs} or the alternative equations \eqref{eq:homothetic_equivalence} are considered, the southern or northern integral can always be expressed as
\begin{align}
E_S(z)& = E_S(\mathcal{X})+\int_{\mathbb{D}}\frac{4}{1+\lvert z\rvert^2}\frac{u^2\lvert \nabla v\rvert^2-2uv\langle \nabla u, \nabla v\rangle +v^2\lvert \nabla u\rvert^2}{(1+\lvert z\rvert^2)^2} dx^1\wedge dx^2,
\end{align}
where $E_S(\mathcal{X})$ is defined in \eqref{equiv:energy}.
\end{theorem}
\begin{proof}
     Since $E_S$ and $E_N$ have the same arguments, we only prove the equivalence for $E_S$. Eqs. \eqref{eq:homothetic_equivalence_u} and \eqref{eq:homothetic_equivalence_v} are the minimal terms in \eqref{eq:Minimal_u} and \eqref{eq:Minimal_v}, respectively, which implies that residual term $\textup{II}$ of \eqref{DiriE:rho} is equal to the minimal term of \eqref{integral}. Plugging the minimal term of \eqref{integral} into \eqref{DiriE:rho} and taking $\rho = 1$, we obtain
\begin{align}
    E_{S}(z) &= \int_{\mathbb{D}}\frac{2}{1+\lvert z\rvert^2}\left\lvert \nabla\ln(1+\lvert z\rvert^2)\right\rvert^2 dx^1\wedge dx^2 \nonumber \\
&\quad +\int_{\mathbb{D}}\frac{4}{1+\lvert z\rvert^2}\frac{u^2\lvert \nabla v\rvert^2-2 u v \langle \nabla u, \nabla v \rangle +v^2\lvert  \nabla u\rvert^2}{(1+\lvert z\rvert^2)^2} dx^1\wedge dx^2 \nonumber \\
&= E_S(\mathcal{X})+\int_{\mathbb{D}}\frac{4}{1+\lvert z\rvert^2}\frac{u^2\lvert \nabla v\rvert^2-2 u v \langle \nabla u, \nabla v \rangle +v^2\lvert  \nabla u\rvert^2}{(1+\lvert z\rvert^2)^2} dx^1\wedge dx^2. \label{eq:Es_ztilde-equiv}
\end{align}
Under the equivalence relation, $E_S(z_c)$ in \eqref{eq:Es_zc-1} with $z_c = u + \mathrm{i} v$ can be expressed as
\begin{align}
E_S(z_c)&= \int_{\mathbb{D}}\frac{1}{1+\lvert z_c\rvert^2}\left\lvert \nabla\ln\left(1+\lvert z_c\rvert^2\right)\right\rvert^2 dx^1\wedge dx^2 \nonumber \\
&\quad +\int_{\mathbb{D}}\frac{4}{1+\lvert z_c\rvert^2}\frac{u^2\lvert \nabla v\rvert^2-2uv\langle \nabla u, \nabla v\rangle +v^2\lvert \nabla u\rvert^2}{\left(1+\lvert z_c\rvert^2\right)^2} dx^1\wedge dx^2 \nonumber\\
&= E_S(\mathcal{X})+\int_{\mathbb{D}}\frac{4}{1+\lvert z_c\rvert^2}\frac{u^2\lvert \nabla v\rvert^2-2uv\langle \nabla u, \nabla v\rangle +v^2\lvert \nabla u\rvert^2}{\left(1+\lvert z_c\rvert^2\right)^2} dx^1\wedge dx^2. \label{eq:Es_zc-equiv}
\end{align}
\end{proof}

\begin{remark} \label{remark_1}
Under the equivalence relations \eqref{sim} and \eqref{equiv:energy}, the energies \eqref{eq:Es_ztilde-equiv} and \eqref{eq:Es_zc-equiv} are the same.
 As long as the minimal conditions \eqref{assumptions} are satisfied, the energies $E_{S}(z)$ and $E_S(z_c)$ decrease to their minimum value $E_S(\mathcal{X})$.
In this argument, we use the north pole stereographic projection at a point $p\in\mathcal{S}$ to flatten a closed genus-zero surface onto $\mathbb{C}$; this mapping is smooth and bijective on $\mathcal{S}$ except at the blow-up point $p$. After the variation on $\mathbb{R}^2$ is used to reduce the Dirichlet energy, there is a one-to-one corresponding variation on $\mathcal{S}$, which decreases the Dirichlet energy. In other words, the surface $\widetilde{\mathcal{S}}$ corresponding to the minimizer of $E\left(\widetilde{w}\right)$ in \eqref{DiriE.5} must be conformally equivalent to a sphere, i.e., the minimizer of $E\left(\widetilde{z}\right)$ in \eqref{DiriE.6} up to a scaling. Assuming that this is not the case, we can find a nonspherical closed genus-zero surface that is conformal to a plane except at one point. This contradicts the uniformization theorem.

Let the area of $\mathcal{S}$ be $4\pi$. As discussed previously, Theorem~\ref{eq:alternative}, the Euler--Lagrange equations \eqref{ELs} and the minimal conditions \eqref{assumptions} correspond to the points varying from $\mathcal{S}$ to $\mathbb{S}^2_{\rho}$ with $\rho = 1/\sqrt{2}$. However, the alternative equations \eqref{eq:homothetic_equivalence} and the minimal conditions \eqref{assumptions} correspond to the points varying from $\mathcal{S}$ to $\mathbb{S}^2$. The benefits of considering the equations \eqref{eq:homothetic_equivalence} are as follows:
\begin{itemize}
\item As long as the Dirichlet energy is decreasing, the decrease can directly tell us by how much the conformality of the associated map is improved for a fixed surface area $4\pi$. 
\item The alternative equations \eqref{eq:homothetic_equivalence} with the minimal conditions \eqref{assumptions} are equivalent to solving the Laplacian equations that will be discussed below.
\end{itemize} 
\end{remark}
\begin{theorem} \label{thm:Laplace_zero}
   Assume that $u$ and $v$ satisfy \eqref{eq:homothetic_equivalence} with the minimal conditions \eqref{assumptions}. Then,
   \begin{align}
       \triangle u = 0, \quad \triangle v = 0, \label{eq:Laplace_zero}
   \end{align}
   except on a measure zero set $\mathcal{Z}:=\{x\in\mathbb{R}^2: u(x)=0\ \mbox{or}\ v(x)=0\}$.
\end{theorem}
\begin{proof}
 First, we consider the case of $u>0$. From \eqref{eq:homothetic_equivalence}, we have that
\begin{align}
-\triangle u &=\frac{2u^2\left(v^2_{x^1}+v^2_{x^2}\right)-2uv\left(u_{x^1}v_{x^1}+u_{x^2}v_{x^2}\right)}{u\left(1+u^2+v^2\right)} \nonumber \\
                &=\frac{ u^2\left(v^2_{x^1}+v^2_{x^2}\right)-2uv\left(u_{x^1}v_{x^1}+u_{x^2}v_{x^2}\right)+v^2\left(u^2_{x^1}+u^2_{x^2}\right)}{u\left(1+u^2+v^2\right)}\quad (\mbox{by \eqref{assumption:a}}) \nonumber \\
                &\geq \frac{2\left\lvert uv\right\rvert \sqrt{\left(v^2_{x^1}+v^2_{x^2}\right)\left(u^2_{x^1}+u^2_{x^2}\right)}-2uv\left(u_{x^1}v_{x^1}+u_{x^2}v_{x^2}\right)}{u\left(1+u^2+v^2\right)} \nonumber \\
                &\geq \frac{2\left\lvert uv\right\rvert \left(\sqrt{\left(v^2_{x^1}+v^2_{x^2}\right)\left(u^2_{x^1}+u^2_{x^2}\right)}-\left(u_{x^1}v_{x^1}+u_{x^2}v_{x^2}\right)\right)}{u\left(1+u^2+v^2\right)}\quad (\mbox{by \eqref{assumption:b}})  \nonumber \\
                &= 0. \label{identity.13}
\end{align}
Hence, $\triangle u = 0$ for $u > 0$.
Next, we consider $u<0$, i.e., $u:=-\hat{u}$, $\hat{u}>0$, which implies $\left\lvert \nabla u\right\rvert =\left\lvert \nabla \hat{u}\right\rvert $.
Then, we rewrite \eqref{eq:homothetic_equivalence} in the form of $\hat{u}$:
\begin{align*} 
\triangle \hat{u}&=- \frac{2\hat{u}v^2_{x^1} + 2\hat{u}v^2_{x^2} - 2\hat{u}_{x^1}v_{x^1}v - 2\hat{u}_{x^2}v_{x^2}v }{ 1+\hat{u}^2+v^2 }.
\end{align*}
Because $u$ and $v$ satisfy the minimal conditions \eqref{assumptions}, $\hat{u}$ and $v$ also satisfy the minimal conditions \eqref{assumptions}. Similar to the derivation of \eqref{identity.13}, we obtain $-\triangle\hat{u} \geq 0$ and then $0 = \triangle\hat{u} = - \triangle u$ for $u < 0$. Under the minimal conditions \eqref{assumptions}, we also have $\triangle v = 0$.
\end{proof}

\section{Convergence of MDEM for spherical conformal parameterizations} \label{sec:3}

A discrete model for a simply connected closed surface is a closed triangular mesh $\mathcal{S}$ of genus zero composed of vertices $\mathcal{V}(\mathcal{S})$, edges $\mathcal{E}(\mathcal{S})$, and faces $\mathcal{F}(\mathcal{S})$. The discrete Dirichlet energy associated with \eqref{DiriE.5} (see, e.g., Chap $10$ in \cite{GuYa08}) is defined by
\begin{subequations} \label{1}
\begin{equation} \label{1a}
E_D(f) = \frac{1}{2} \mathrm{trace}(\mathbf{f}^\top L_D \mathbf{f}),
\end{equation}
where $\mathbf{f}\in\mathbb{R}^{n\times 3}$ and $L_D \in \mathbb{R}^{n \times n}$ is a Laplacian matrix with
\begin{equation} \label{1b}
[L_D]_{ij} = 
\begin{cases}
w_{ij} \equiv  -\frac{1}{2} (\cot\alpha_{ij} + \cot \alpha_{ji}) & \text{if $[v_i,v_j]\in\mathcal{E}(\mathcal{S})$}, \\
-\sum_{\ell\neq i} w_{i\ell} & \text{if $i=j$}, \\
0 & \text{otherwise},
\end{cases}
\end{equation}
\end{subequations}
in which $n=\#\mathcal{V}(\mathcal{S})$ and $\alpha_{ij}$ and $\alpha_{ji}$ are two angles opposite to the edge $[v_i,v_j]$. 
A spherical conformal map from $\mathcal{V}(\mathcal{S})$ to $\mathbb{S}^2$ is to solve the optimization problem
\begin{equation}
\label{mintr}
    \min\left\{\text{trace}(\mathbf{f}^{\top}L_D\mathbf{f})\mid \mathbf{f}\in\mathbb{R}^{n\times 3}: \mathcal{V}(\mathcal{S})\to \mathbb{S}^2\right\}.
\end{equation}
The corresponding spherical conformal function from $\mathcal{S}$ to $\mathbb{S}^2$ is a piecewise linear function on $\mathcal{F}(\mathcal{S})$ determined by the barycentric coordinates on $\mathcal{V}(\mathcal{S})$.

As mentioned in Section \ref{sec:Intro}, since $\mathbb{S}^2$ is not convex, it is difficult to use the gradient projection method or the heat diffusion flow method for solving the optimization problem \eqref{mintr} with a constraint on $\mathbb{S}^2$ to guarantee convergence.

To address this, utilizing the stereographic projection to transform $\mathbb{S}^2$ onto $\overline{\mathbb{C}}$, a DEM was first proposed in \cite{YuLL19} and designed for the computation of a spherical conformal parameterization between $\mathcal{S}$ and $\mathbb{S}^2$ by alternatingly solving Laplacian equations \eqref{eq:Laplace_zero} corresponding to southern and northern hemispheres with various sizes, which may cause some difficulty in the mathematical proof of convergence.

Considering the theoretical foundation of the Dirichlet energy minimization on $\overline{\mathbb{C}}$ in Section \ref{sec:Thm_Found_DEM}, in this section, we propose an MDEM algorithm with a nonequivalence deflation technique for the computation of spherical conformal parameterizations of simply connected closed triangular meshes. In addition, we prove the asymptotically R-linear convergence of the MDEM algorithm.

We first propose the MDEM algorithm, which is slightly different from the DEM in \cite{YuLL19}. Let 
\begin{equation*}
    \mathbf{f}^{(1)}=[f_1^{(1)}, \cdots, f_n^{(1)}]^\top \in\mathbb{R}^{n\times 3}\quad \mbox{with $f_i^{(1)}\equiv [f_{i,1}^{(1)}, f_{i,2}^{(1)}, f_{i,3}^{(1)}]\in\mathbb{S}^2$}
\end{equation*} be an initial spherical conformal parameterization of $\mathcal{S}$ constructed by the method in \cite{AnHT99}. Let
\begin{align}
    {h}_i^{(1)} = \Pi_{\mathbb{S}^2}(f_i^{(1)})\equiv  \frac{{f}_{i,1}^{(1)}}{1-{f}_{i,3}^{(1)}} + \sqrt{-1}\frac{{f}_{i,2}^{(1)}}{1-{f}_{i,3}^{(1)}} \label{eq:ster_proj}
\end{align}
be the stereographic projection from ${f}_i^{(1)} \in \mathbb{S}^2$ onto ${h}_i^{(1)} \in \overline{\mathbb{C}}$. Given a radius $\rho\gtrsim 1$, we define
\begin{align}
\mathtt{I}_s=\{i\mid\lvert {h}_i^{(s)}\rvert < \rho\}, \ \mathtt{B}_s^{\prime}=\{1, \dots,n\}\backslash\mathtt{I}_s, \  \mathtt{B}_s = \{ j \mid [v_i, v_j ] \in \mathcal{E}(\mathcal{S}), i \in \mathtt{I}_s, j \in \mathtt{B}_s^{\prime}\}, \label{eq:idx_sets_I_Bprime_B}
\end{align}
$n_s = \# \mathtt{I}_s$ and $m_s = \# \mathtt{B}_s$, where the index $s$ indicates the southern and northern hemispheres.
In view of Theorem~\ref{thm:Laplace_zero}, for the minimization of the Dirichlet energy corresponding to the southern and northern hemispheres, the discrete Laplacian equation $L_D {\mathbf{h}} = 0$ must be solved alternatingly; i.e., for $s=1, 2,$ 
\begin{align}
      \begin{bmatrix}
            L_s & B_s & 0 \\
            B_s^{\top} & [L_D]_{\mathtt{B}_s,\mathtt{B}_s} & [L_D]_{\mathtt{B}_s, \mathtt{B}_s^c} \\
            0 & [L_D]_{\mathtt{B}_s^c,\mathtt{B}_s} & [L_D]_{\mathtt{B}_s^c, \mathtt{B}_s^c}
      \end{bmatrix}  \begin{bmatrix}
            \mathbf{h}_{\mathtt{I}_s} \\ \mathbf{h}_{\mathtt{B}_s} \\ \mathbf{h}_{\mathtt{B}_s^c} 
      \end{bmatrix} = 0 \label{eq:Laplacian_Sys}
\end{align}
with $\mathtt{B}_s^c = \mathtt{B}_s^{\prime}\backslash \mathtt{B}_s$, and
\begin{align} \label{2a}
L_s = [L_D]_{\mathtt{I}_s,\mathtt{I}_s} \in \mathbb{R}^{n_s\times n_s}, \quad B_s = [L_D]_{\mathtt{I}_s,\mathtt{B}_s} \in \mathbb{R}^{n_s\times m_s}.
\end{align}
For the iteration of the southern hemisphere, we set $\mathbf{h}_{\mathtt{B}_1^{\prime}} = [{h}_i^{(1)}]_{\mathtt{B}_1^{\prime}}$ in the first equation of \eqref{eq:Laplacian_Sys} with $s=1$ and solve
\begin{align}
      L_1 \mathbf{h}_{\mathtt{I}_1} = - B_1 \mathbf{h}_{\mathtt{B}_1}. \label{eq:Laplacian_LS_south}
\end{align}
Then, the inversion of $\mathbf{h}$ on $\overline{\mathbb{C}}$ is taken as
\begin{align}
      \mathbf{h}^{(2)} = \begin{bmatrix}
          {h}_1^{(2)} & \cdots & {h}_n^{(2)}
      \end{bmatrix}^{\top} \equiv  \mathrm{diag}(\lvert\mathbf{h}\rvert)^{-2}\mathbf{h} \label{eq:inversion_h}
\end{align}
for the northern hemisphere. We further set $\mathbf{h}_{\mathtt{B}_2^{\prime}} = \mathbf{h}^{(2)}_{\mathtt{B}_2^{\prime}}$ and solve the first equation of \eqref{eq:Laplacian_Sys} with $s=2$,
\begin{align}
      L_2 \mathbf{h}_{\mathtt{I}_2} = - B_2 \mathbf{h}_{\mathtt{B}_2}, \label{eq:Laplacian_LS_north}
\end{align}
for the iteration of the northern hemisphere. 

By taking the inversion of $\mathbf{h}$ again for the southern hemisphere, we repeat the southern and northern hemisphere iterations in \eqref{eq:Laplacian_LS_south} and \eqref{eq:Laplacian_LS_north} alternatingly until convergence. The DEM algorithm with fixed Laplacian matrix sizes $L_s$ in \eqref{eq:Laplacian_Sys} for the computation of spherical conformal parameterizations is stated as Algorithm~\ref{alg1}.

\begin{algorithm}
\caption{DEM for spherical conformal parameterizations}
\label{alg1}
\begin{algorithmic}[1]
\Require A closed triangular mesh $\mathcal{S}$ of genus zero and a radius $\rho \gtrsim 1$.
\Ensure A spherical conformal parameterization $\mathbf{f}$.
\State Compute an initial spherical parameterization $\mathbf{f}$ using \cite{AnHT99}. 
\State Compute the stereographic projection ${h}_{i}=\frac{{f}_{i,1}}{1- {f}_{i,3}}+\sqrt{-1}\frac{{f}_{i,2}}{1- {f}_{i,3}}$ for all $i$. 
\Repeat
\For{$s=1,2$} \label{alg:CEM_for_south_north}
\State Take the inversion $\mathbf{h}\gets\mathrm{diag}(\lvert\mathbf{h}\rvert)^{-2}\mathbf{h}$. \label{alg1:4} 
\State Take the index sets $\mathtt{I}_s$, $\mathtt{B}_s^{\prime}$ and $\mathtt{B}_s$ in \eqref{eq:idx_sets_I_Bprime_B}. \label{alg1:6}
\State Set up  $L_s$ and $B_s$  in \eqref{2a}.
\State Solve the linear system
$L_s\mathbf{h}_{\mathtt{I}_s} = -B_s  \mathbf{h}_{\mathtt{B}_s}$. \label{alg1:8}
\EndFor \label{alg:CEM_endfor_south_north}
\State Take the inverse stereographic projection ${f}_i\gets\Pi_{\mathbb{S}^2}^{-1}({h}_i)$ for all $i$. 
\State Compute the Dirichlet energy $E_D(f)$ in \eqref{1a}. 
\Until{Dirichlet energy converges}
\State \textbf{return} the spherical conformal parameterization $\mathbf{f}$.
\end{algorithmic}
\end{algorithm}

The numerical results in \cite{YuLL19} show the convergence of the DEM algorithm; however, a theoretical proof of convergence is still lacking. The index sets $\mathtt{I}_s$ and $\mathtt{B}_s$ for $s = 1, 2$ of \eqref{eq:idx_sets_I_Bprime_B} in \cite{YuLL19} change in each iteration, which increases the difficulty of proving convergence.

In what follows, we propose an efficient and reliable MDEM algorithm with nonequivalence deflation and prove its asymptotically R-linear convergence.

To fix the index sets in \eqref{eq:idx_sets_I_Bprime_B}, from \eqref{eq:Laplacian_LS_south}--\eqref{eq:Laplacian_LS_north}, 
 the $k$th southern and northern hemisphere iterations for $\mathbf{h}_{\mathtt{I}_s}$ and $\mathbf{h}_{\mathtt{B}_s}$ can be written as
\begin{subequations}
\begin{equation} \label{3}
\mathbf{h}_{\mathtt{I}_s}^{(k)} = -L_s^{-1}B_s\mathbf{h}_{\mathtt{B}_s}^{(k)} ~~\text{ and }~~
\mathbf{h}_{\mathtt{B}_{s+1}}^{(k+1)} = \mathrm{diag}(\lvert P_s\mathbf{h}_{\mathtt{I}_s}^{(k)}\rvert)^{-2} P_s\mathbf{h}_{\mathtt{I}_s}^{(k)}, \quad s=1, 2,
\end{equation}
where
\begin{equation} \label{2b}
P_1 = [I_n]_{\mathtt{B}_{2},\mathtt{I}_1} \in \mathbb{R}^{m_2\times n_1},  \quad P_2 = [I_n]_{\mathtt{B}_{1},\mathtt{I}_2} \in \mathbb{R}^{m_1\times n_2}.
\end{equation}
\end{subequations} 
For convenience, denoting $\mathbf{h}_s^{(k)}:=\mathbf{h}_{\mathtt{B}_s}^{(k)}$, the iterations in \eqref{3} can be simplified to
\begin{subequations} \label{4}
\begin{align} 
\mathbf{h}_2^{(k+1)} &= \mathrm{diag}(\lvert A_1 \mathbf{h}_1^{(k)}\rvert)^{-2} A_1 \mathbf{h}_1^{(k)}, \label{4a}\\
\mathbf{h}_1^{(k+2)} &= \mathrm{diag}(\lvert A_2 \mathbf{h}_2^{(k+1)}\rvert)^{-2} A_2 \mathbf{h}_2^{(k+1)}, \label{4b}
\end{align}
where $A_1 \in \mathbb{R}^{m_2 \times m_1}$ and $A_2 \in \mathbb{R}^{m_1 \times m_2}$ with
\begin{equation} \label{2c}
A_s = -P_s L_s^{-1} B_s, ~s=1,2.
\end{equation}
\end{subequations}

In discussing the convergence of the MDEM algorithm, we consider the critical issue for the convergence of iterations in \eqref{4}. If we ignore the diagonal matrices in \eqref{4a} and \eqref{4b}, then $\mathbf{h}_1^{(k+2)}$ can be written as 
\begin{align}
\mathbf{h}_1^{(k+2)} = A_2 A_1 \mathbf{h}_1^{(k+1)}. \label{eq:simp_h1} 
\end{align}
Let $\mathbf{1}_{\ell} = [ 1, \cdots, 1]^{\top} \in \mathbb{R}^{\ell}$. From \eqref{1b} and \eqref{2a}, we have
\begin{align*}
      L_s \mathbf{1}_{n_s} + B_s \mathbf{1}_{m_s} = 0, \quad \mbox{i.e., } \quad \mathbf{1}_{n_s} = - L_s^{-1} B_s \mathbf{1}_{m_s},
\end{align*}
which implies that
\begin{align}
      A_2 A_1 \mathbf{1}_{m_1} = A_2 P_1 \mathbf{1}_{n_1} = A_2 \mathbf{1}_{m_2} = P_2 \mathbf{1}_{n_2} =  \mathbf{1}_{m_1}. \label{eq:ew_1_A2A1}
\end{align}
Since $L_s$, $s=1,2$, are M-matrices (i.e., $L_s^{-1} \geqslant 0$), it is easily seen that $A_1 \geq 0$, $A_2 \geq 0$, and $A_2 A_1 \geq 0$ are nonnegative and $A_2 A_1$ is irreducible. From the Perron--Frobenius theorem \cite[Chap. 8]{meye:2000} and \eqref{eq:ew_1_A2A1}, we have
$\rho(A_2A_1) = 1$, which implies that the convergence of the iteration in \eqref{eq:simp_h1} cannot be guaranteed. Consequently, to ensure the convergence of \eqref{eq:simp_h1}, a nonequivalence deflation technique to move the eigenvalue $1$ to zero is needed.

Let $[\mathbf{q}_1^{\top}, \mathbf{q}_2^{\top}]^{\top}$ with $\mathbf{q}_1 \in \mathbb{R}^{m_2}$ and $\mathbf{q}_2 \in \mathbb{R}^{m_1}$ be the left eigenvector of
\begin{align}
     \mathcal{A} \equiv \begin{bmatrix} 0 & A_1 \\ A_2 & 0 \end{bmatrix} \label{eq:mtx_A}
\end{align}
corresponding to the eigenvalue $\lambda$. That is,
\begin{align}
      \mathbf{q}_2^{\top} A_2 = \lambda \mathbf{q}_1^{\top}, \quad \mathbf{q}_1^{\top} A_1 = \lambda \mathbf{q}_2^{\top} \ \Rightarrow \ \mathbf{q}_2^{\top} A_2 A_1 = \lambda^2 \mathbf{q}_2^{\top}, \label{eq:left_ev_A2A1}
\end{align}
which implies that $\mathcal{A}$ has eigenvalues $\pm 1$.

\begin{theorem} \label{thm:noonequivalent_deflation}
Let $[\mathbf{q}_1^{\top}, \mathbf{q}_2^{\top}]^{\top}$ with $\mathbf{q}_1 \in \mathbb{R}^{m_2}$, $\mathbf{q}_2 \in \mathbb{R}^{m_1}$ and $\mathbf{q}_2^{\top} \mathbf{1}_{m_1} = 1$ be the left eigenvector of $\mathcal{A}$ 
in \eqref{eq:mtx_A} corresponding to the eigenvalue $1$. Then, 
\begin{align}
\sigma((A_2 - \mathbf{1}_{m_1} \mathbf{q}_1^{\top})A_1) = \sigma(A_2 A_1 - \mathbf{1}_{m_1} \mathbf{q}_2^{\top})= (\sigma(A_2 A_1) \backslash \{ 1\}) \cup \{ 0 \} \label{eq:spectrum_deflated}
\end{align} 
and $\rho((A_2 - \mathbf{1}_{m_1} \mathbf{q}_1^{\top}) A_1) = \rho(A_2 A_1 - \mathbf{1}_{m_1} \mathbf{q}_2^{\top}) < 1$. 
\end{theorem}
\begin{proof}
From \eqref{eq:left_ev_A2A1} with $\lambda = 1$, $\mathbf{q}_2$ is the left eigenvector of $A_2 A_1$ corresponding to the eigenvalue $1$. Since $A_2 A_1 \geq 0$ are nonnegative and irreducible, from the Perron--Frobenius theorem, we have $\mathbf{q}_2^{\top} > 0$. Then, $\mathbf{q}_1^{\top} > 0$, from $A_2 \geq 0$ and \eqref{eq:left_ev_A2A1}. By the assumption that $\mathbf{q}_2^{\top} \mathbf{1}_{m_1} = 1$, we have that
\begin{align*}
     (A_2 A_1 - \mathbf{1}_{m_1} \mathbf{q}_2^{\top}) \mathbf{1}_{m_1} = 0.
\end{align*}
Let $\mathbf{p}$ be the right eigenvector of $A_2 A_1$ corresponding to the eigenvalue $\lambda \neq 1$. Then, $\mathbf{q}_2^{\top} \mathbf{p} = 0$ and $(A_2 A_1 - \mathbf{1}_{m_1} \mathbf{q}_2^{\top}) \mathbf{p} = \lambda \mathbf{p}$, which implies the result in \eqref{eq:spectrum_deflated}. This completes the proof of this theorem.
\end{proof}

Using the nonequivalence deflation in Theorem~\ref{thm:noonequivalent_deflation}, the iterations in \eqref{4} can be modified to
\begin{subequations}  \label{eq:nonequ_defl_iter}
\begin{align} 
\mathbf{h}_2^{(k+1)} &= \mathrm{diag}(\lvert A_1 \mathbf{h}_1^{(k)}\rvert)^{-2} A_1 \mathbf{h}_1^{(k)}, \label{eq:iter_h2}\\
\mathbf{h}_1^{(k+2)} &= \mathrm{diag}(\lvert \widehat{A}_2 \mathbf{h}_2^{(k+1)}\rvert)^{-2} \widehat{A}_2 \mathbf{h}_2^{(k+1)}, \label{eq:defl_iter_h1}
\end{align} 
where
\begin{align}
      \widehat{A}_2 = A_2 - \mathbf{1}_{m_1}\mathbf{q}_1^{\top}. \label{eq:A2_tilde}
\end{align}
\end{subequations}
In light of the equivalence representation in Theorem \ref{eq:alternative} for the energies $E_S(z)$ in \eqref{DiriEs.8-2-1} and $E_N(\zeta)$ in \eqref{identity.1}, the minimizer of $E\left(\widetilde{z}\right)$ in \eqref{Dirichlet:Z} is equivalent up to a scaling, so we modify \eqref{eq:nonequ_defl_iter} to
\begin{subequations} \label{eq:scaling_iteration}
\begin{align} 
\mathbf{h}_2^{(k+1)} &= c_k \mathrm{diag}(\lvert A_1 \mathbf{h}_1^{(k)}\rvert)^{-2} A_1 \mathbf{h}_1^{(k)}, \label{eq:scaling_iteration_s}\\
\mathbf{h}_1^{(k+2)} &= c_{k+1} \mathrm{diag}(\lvert \widehat{A}_2 \mathbf{h}_2^{(k+1)}\rvert)^{-2} \widehat{A}_2 \mathbf{h}_2^{(k+1)}, \label{eq:scaling_iteration_n}
\end{align}
where
\begin{align}
     c_k = \max_{1 \leq i \leq m_2} \{ \lvert  A_1 \mathbf{h}_1^{(k)} \rvert_i \}, \quad c_{k+1} = \max_{1 \leq i \leq m_1} \{ \lvert  \widehat{A}_2 \mathbf{h}_2^{(k+1)} \rvert_i \} \label{eq:scaling}
\end{align}
\end{subequations}
such that $\lvert  A_1 \mathbf{h}_1^{(k)}  / c_k \rvert_i \leq 1$ and $\lvert \widehat{A}_2 \mathbf{h}_2^{(k+1)}  / c_{k+1} \rvert_i \leq 1$ for all $i$.

When the iterations of \eqref{eq:scaling_iteration} are convergent, one swap of the southern and northern hemisphere iterations without scaling is used to reconstruct $\mathbf{h}$. Then, the inverse stereographic projection for $\mathbf{h}$ can be applied to construct the spherical conformal parameterization $\mathbf{f}$. We summarize the MDEM with nonequivalence deflation for the computation of the spherical conformal map for $\mathcal{S}$ in Algorithm \ref{alg:CEM_defl}.

Now, we show that the MDEM with nonequivalence deflation converges asymptotically and R-linearly under some mild conditions. Based on the equivalence relations \eqref{eq:Es_ztilde-equiv} and \eqref{eq:Es_zc-equiv}, the scalings $c_k$ and $c_{k+1}$ in \eqref{eq:scaling_iteration} do not affect the minimization; for simplicity, we show the convergence of the iterations in \eqref{eq:nonequ_defl_iter}.
Suppose $\mathbf{f}^{(*)}$ is the unique minimizer of $\eqref{1a}$ and that the corresponding boundary points $\mathbf{h}_{\mathtt{B}_1}^{(*)}$ and $\mathbf{h}_{\mathtt{B}_2}^{(*)}$ satisfy
\begin{subequations}
\begin{align} 
\mathbf{h}_{\mathtt{B}_2}^{(*)}&\equiv \mathbf{h}_2^{(*)} = \mathrm{diag}(\lvert A_1 \mathbf{h}_1^{(*)}\rvert)^{-2} A_1 \mathbf{h}_1^{(*)}, \label{5a}\\
\mathbf{h}_{\mathtt{B}_1}^{(*)}&\equiv \mathbf{h}_1^{(*)} = \mathrm{diag}(\lvert \widehat{A}_2 \mathbf{h}_2^{(*)}\rvert)^{-2} \widehat{A}_2 \mathbf{h}_2^{(*)}. \label{5b}
\end{align}
\end{subequations}
In what follows, we aim to show $\mathbf{h}_s^{(k)}\to\mathbf{h}_s^{(*)}$, $s=1,2$, to be asymptotically R-linear convergent.

\begin{theorem} \label{thm:convergence}
Let $\mathcal{S}$ be a closed surface of genus zero with a normalized area of one that is equipped with a Delaunay triangular mesh. 
Suppose that there exists a sufficiently small $\eta > 0$ such that
\begin{subequations} \label{eq:assumption_convergence}
\begin{align}
    1-\eta\leqslant\lvert A_1\mathbf{h}_1^{(k)}\rvert_i\leqslant 1, \quad 1-\eta\leqslant\lvert \widehat{A}_2\mathbf{h}_2^{(k)}\rvert_i\leqslant 1 \label{eq:assumption_convergence_1}
\end{align}
for all $i$, and 
\begin{align}
\rho(\gamma^2 \lvert \widehat{A}_2 \rvert A_1) <1, \label{eq:assumption_convergence_2}
\end{align}
\end{subequations}
where $\gamma = 1/(1 - \eta)^2$; then, we have $k^*\in\mathbb{N}$ that is sufficiently large and $r^*<1$ such that $\|\mathbf{h}_s^{(k)}-\mathbf{h}_s^{(*)}\|_\infty^{1/k}\leqslant r^* <1$ for all $k\geqslant k^*$, $s=1,2$; i.e., $\mathbf{h}_s^{(k)}\to\mathbf{h}_s^{(*)}$ converges asymptotically and R-linearly.
\end{theorem}
\begin{proof}
Let $\varepsilon_s^{(k)}=\mathbf{h}_s^{(k)}-\mathbf{h}_s^{(*)}$, $s=1,2$. From \eqref{4a} and \eqref{5a}, we have
\begin{align*}
\bm{\varepsilon}_2^{(k+1)} &= \mathbf{h}_2^{(k+1)}-\mathbf{h}_2^{(*)} = \mathrm{diag}(\lvert A_1 \mathbf{h}_1^{(k)}\rvert)^{-2} A_1 \mathbf{h}_1^{(k)} - \mathrm{diag}(\lvert A_1 \mathbf{h}_1^{(*)}\rvert)^{-2} A_1 \mathbf{h}_1^{(*)}. 
\end{align*}
Since
\begin{align*}
     \mathbf{e}_i^{\top} \varepsilon_2^{(k+1)} &= \mathbf{e}_i^{\top}\mathrm{diag}(\lvert A_1 \mathbf{h}_1^{(k)}\rvert)^{-2} A_1 \mathbf{h}_1^{(k)} - \mathbf{e}_i^{\top} \mathrm{diag}(\lvert A_1 \mathbf{h}_1^{(*)} \rvert)^{-2} A_1 \mathbf{h}_1^{(*)} \\
     &= (\mathbf{e}_i^{\top}\lvert A_1 \mathbf{h}_1^{(k)} \rvert)^{-2} (\mathbf{e}_i^{\top}A_1 \mathbf{h}_1^{(k)}) - (\mathbf{e}_i^{\top} \lvert A_1 \mathbf{h}_1^{(*)} \rvert)^{-2}(\mathbf{e}_i^{\top} A_1 \mathbf{h}_1^{(*)}) \\
     &= \left( \mbox{conj}(\mathbf{e}_i^{\top}A_1 \mathbf{h}_1^{(k)})\right)^{-1} - \left(\mbox{conj}(\mathbf{e}_i^{\top} A_1 \mathbf{h}_1^{(*)}) \right)^{-1}\\
     &= \left( \mbox{conj}(\mathbf{e}_i^{\top}A_1 \mathbf{h}_1^{(k)} )\mbox{conj}(\mathbf{e}_i^{\top} A_1 \mathbf{h}_1^{(*)})\right)^{-1} \left( \mbox{conj}(\mathbf{e}_i^{\top} A_1 \mathbf{h}_1^{(*)}-\mathbf{e}_i^{\top}A_1 \mathbf{h}_1^{(k)})\right)\\
     &= \left( \mbox{conj}(\mathbf{e}_i^{\top}A_1 \mathbf{h}_1^{(k)} )\mbox{conj}(\mathbf{e}_i^{\top} A_1 \mathbf{h}_1^{(*)})\right)^{-1}\left( \mbox{conj}(\mathbf{e}_i^{\top} A_1 (\mathbf{h}_1^{(*)} -  \mathbf{h}_1^{(k)}))\right),
\end{align*}
we have that
\begin{align}
\mbox{conj}(\varepsilon_2^{(k+1)})  
&= - \mathrm{diag}\left( \frac{1}{({\mathbf{e}_i^{\top}} A_1 \mathbf{h}_1^{(k)}) (\mathbf{e}_i^{\top} A_1 \mathbf{h}_1^{(*)})}\right) A_1 \varepsilon_1^{(k)}. \label{9}
\end{align}
Similarly,
\begin{align}
& \varepsilon_1^{(k+2)} \nonumber \\ 
=&\ - \mathrm{diag}\left( \frac{1}{\mbox{conj}(({\mathbf{e}_i^{\top}} \widehat{A}_2 \mathbf{h}_2^{(k+1)}) (\mathbf{e}_i^{\top} \widehat{A}_2 \mathbf{h}_2^{(*)}))}\right) \widehat{A}_2 \mbox{conj} \left( \varepsilon_2^{(k+1)} \right)  \nonumber \\
=&  \mathrm{diag}\left( \frac{1}{\mbox{conj}(({\mathbf{e}_i^{\top}} \widehat{A}_2 \mathbf{h}_2^{(k+1)}) (\mathbf{e}_i^{\top} \widehat{A}_2 \mathbf{h}_2^{(*)}))}\right) \widehat{A}_2 \mathrm{diag} \left( \frac{1}{({\mathbf{e}_i^{\top}} A_1 \mathbf{h}_1^{(k)}) (\mathbf{e}_i^{\top} A_1 \mathbf{h}_1^{(*)})}\right) A_1 \varepsilon_1^{(k)}. \label{eq:iter_eps_1}
\end{align}
Since $A_1=-P_1L_1^{-1}B_1\geqslant 0$, taking the componentwise absolute values of the vectors in \eqref{eq:iter_eps_1}, it holds from the assumption that 
\begin{equation} \label{12}
\lvert\bm{\varepsilon}_1^{(k+2)}\rvert \leqslant \gamma^2 \lvert \widehat{A}_2\rvert A_1 \lvert\bm{\varepsilon}_1^{(k)}\rvert \leqslant (\gamma^2 \lvert \widehat{A}_2\rvert A_1)^{k/2} \lvert\bm{\varepsilon}_1^{(0)}\rvert.
\end{equation}
From the assumption of \eqref{eq:assumption_convergence_2}, there is an operator norm $\| \cdot \|_*$ such that 
\begin{align}
    \| \gamma^2 \lvert \widehat{A}_2 \rvert A_1 \|_{*} < 1. \label{eq:pf_R-linear_5}
\end{align}
From \eqref{12}, it follows that
\begin{align*}
\|\bm{\varepsilon}_1^{(k+2)}\|_\infty 
&\leqslant \|(\gamma^2 \lvert \widehat{A}_2\rvert A_1)^{k/2}\|_\infty \|\bm{\varepsilon}_1^{(0)}\|_\infty
\leqslant M_{\infty} \|(\gamma^2 \lvert \widehat{A}_2\rvert A_1)^{k/2}\|_{*} \|\bm{\varepsilon}_1^{(0)}\|_\infty 
\end{align*}
for some constant $M_{\infty}>0$. Then, by \eqref{eq:pf_R-linear_5}, we have a $k^*\in\mathbb{N}$ and $0<r^*<1$ such that 
$$
\|\bm{\varepsilon}_1^{(k+2)}\|_\infty^{1/k} \leqslant M_{\infty}^{1/k} \|\gamma \lvert \widehat{A}_2\rvert A_1\|_*^{1/2} \leqslant r^* < 1,   ~\text{ $\forall k\geq k^*$.} 
$$
Similarly, we have
$$
\|\bm{\varepsilon}_2^{(k+3)}\|_\infty^{1/k} \leqslant M_{\infty}^{1/k} \|\gamma^2 A_1 \lvert \widehat{A}_2\rvert\|_*^{1/2} \leqslant r^* <1,   ~\text{ $\forall k\geqslant k^*$.} 
$$
\end{proof}

\begin{algorithm}
\caption{MDEM with nonequivalent deflation for spherical conformal parameterizations}
\label{alg:CEM_defl}
\begin{algorithmic}[1]
\Require A closed triangular mesh $\mathcal{S}$ of genus zero, a radius $\rho \approx 1.4$ and $n=\#\mathcal{V}(\mathcal{S})$. 
\Ensure A spherical conformal parameterization $\mathbf{f}$.
\State Compute an initial spherical parameterization $\mathbf{f}$ using \cite{AnHT99}. 
\State Compute the stereographic projection ${h}_i=\frac{{f}_{i,1}}{1- {f}_{i,3}}+\mathrm{i}\frac{{f}_{i,2}}{1-{f}_{i,3}}$, $i=1, \dots, n$. 
\State \% Use the following steps of the southern and northern hemisphere iterations to define the matrices $L_s, B_s, P_s$ in \eqref{2a} and \eqref{2b}, respectively.
\For{$s=1,2$}  
\State Take the inversion $\mathbf{h}\gets\mathrm{diag}(\lvert\mathbf{h}\rvert)^{-2}\mathbf{h}$.  
\State Take the index sets $\mathtt{I}_s$, $\mathtt{B}_s^{\prime}$ and $\mathtt{B}_s$ in \eqref{eq:idx_sets_I_Bprime_B}, and set $m_s = \#\mathtt{B}_s$.  
\State Set up $L_s$ and $B_s$ in \eqref{2a}.
\State Solve the linear system
$L_s\mathbf{h}_{\mathtt{I}_s} = -B_s  \mathbf{h}_{\mathtt{B}_s}$.  
\EndFor 
\State Set $P_1 = [I_n]_{\mathtt{B}_{2},\mathtt{I}_1}$, $P_2 = [I_n]_{\mathtt{B}_{1},\mathtt{I}_2}$.
\State Compute the left eigenvector $[\mathbf{q}_1^{\top}, \mathbf{q}_2^{\top}]^{\top}$ in Theorem~\ref{thm:noonequivalent_deflation}.
\Repeat 
\State Compute the scaling $c_k$ in \eqref{eq:scaling} and $\mathbf{h}_{\mathtt{B}_2}$ in \eqref{eq:scaling_iteration_s}. 
\State Compute the scaling $c_{k+1}$ in \eqref{eq:scaling} and $\mathbf{h}_{\mathtt{B}_1}$ in \eqref{eq:scaling_iteration_n}. 
\Until{$\mathbf{h}_{\mathtt{B}_s}$, $s=1,2$, converge}  
\State \% Reconstruct $\mathbf{h}$ by the following steps for the southern- and northern-hemisphere iterations:
\State Solve $L_1\mathbf{h}_{\mathtt{I}_1} = -B_1  \mathbf{h}_{\mathtt{B}_1}$, and take the inversion $\mathbf{h}\gets\mathrm{diag}(\lvert\mathbf{h}\rvert)^{-2}\mathbf{h}$.
\State Solve $L_2\mathbf{h}_{\mathtt{I}_2} = -B_2  \mathbf{h}_{\mathtt{B}_2}$.
\State Normalize $\mathbf{h} \gets \mathbf{h} / \mathrm{median}(\lvert \mathbf{h}\rvert)$. 
\State Construct the spherical conformal parameterization $\mathbf{f}$ by an inverse stereographic projection of $\mathbf{h}$: $f_{\ell} = \Pi_{\mathbb{S}^2}^{-1}(h_{\ell})$, $\ell=1, \dots, n$.
\end{algorithmic}
\end{algorithm}

\section{Numerical Experiments} \label{sec:4}

\begin{figure}
\center
\begin{subfigure}[b]{0.28\textwidth}
\center
    \includegraphics[width=\textwidth]{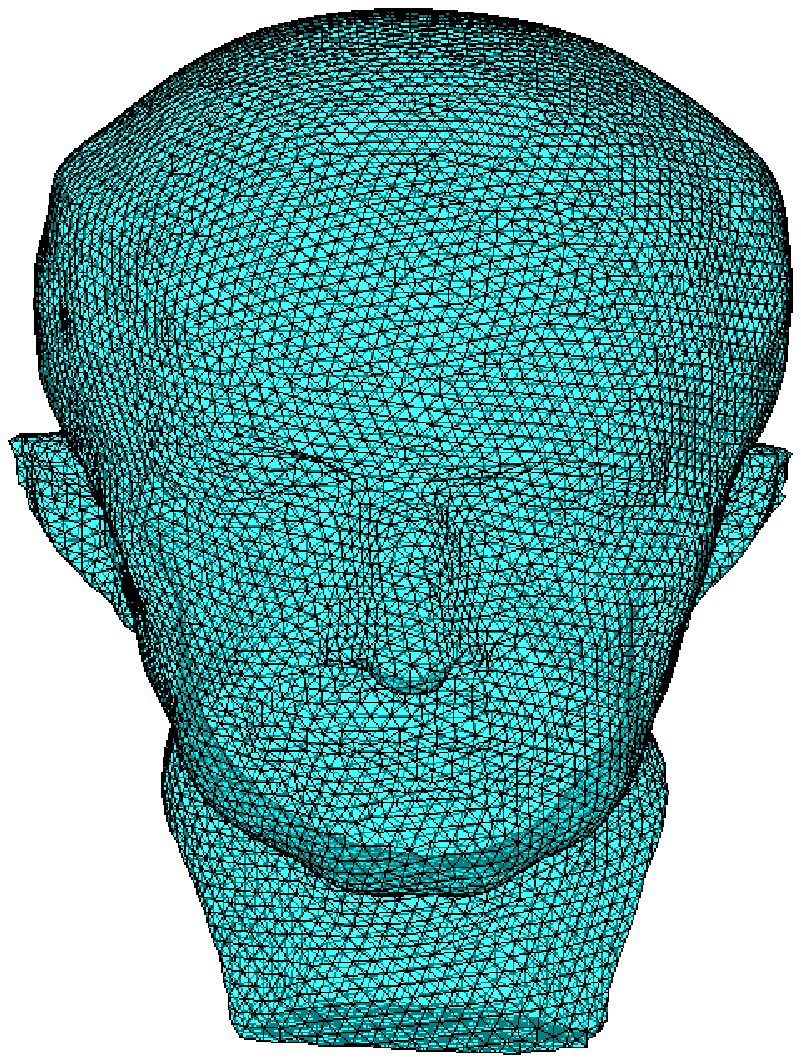}
\caption{Max Planck}
    \label{fig:MaxPlanck_mesh}
\end{subfigure}
\begin{subfigure}[b]{0.23\textwidth}
\center
    \includegraphics[width=\textwidth]{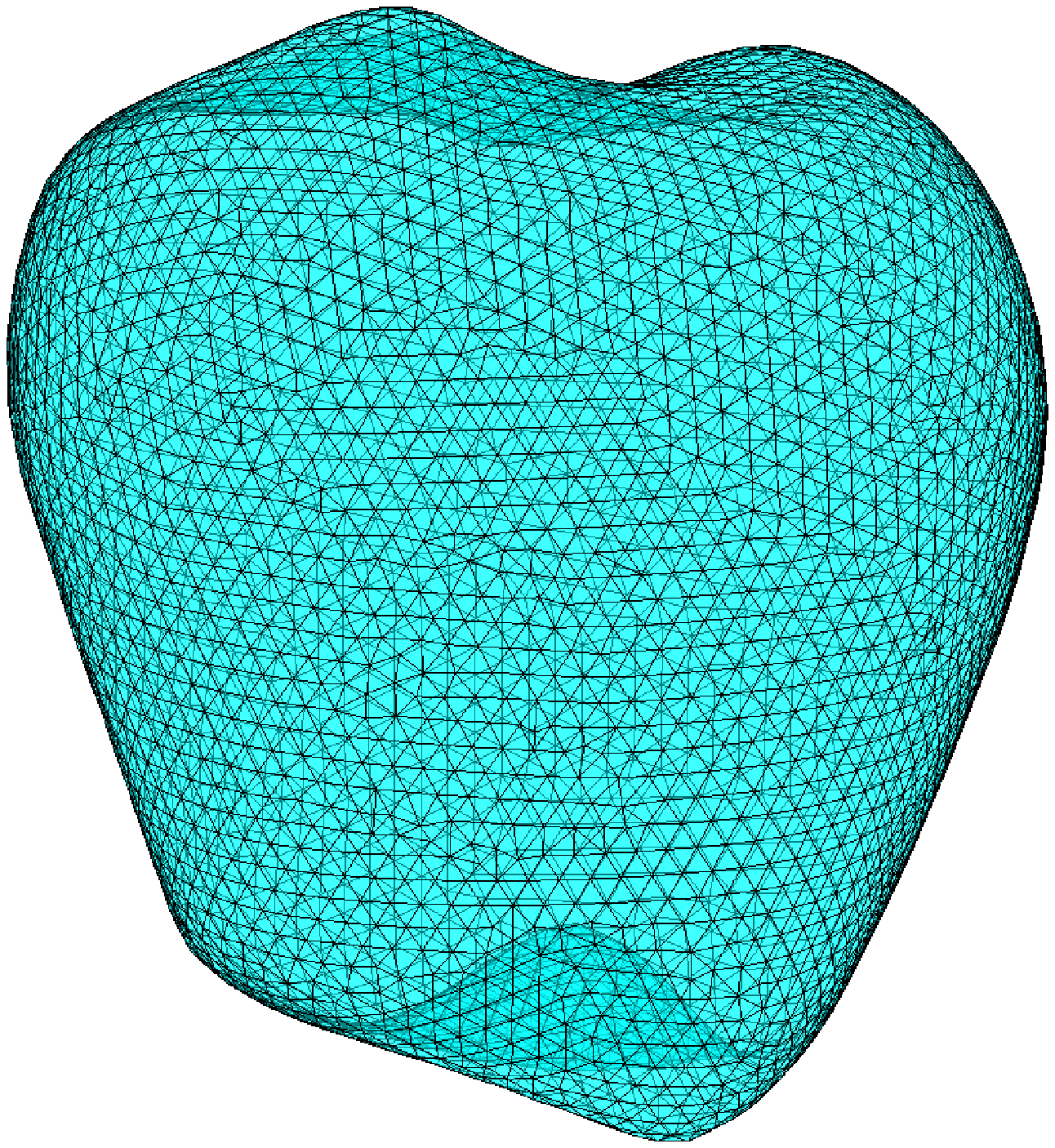}
\caption{Apple}
    \label{fig:Apple_mesh}
\end{subfigure}
\begin{subfigure}[b]{0.23\textwidth}
\center
    \includegraphics[width=\textwidth]{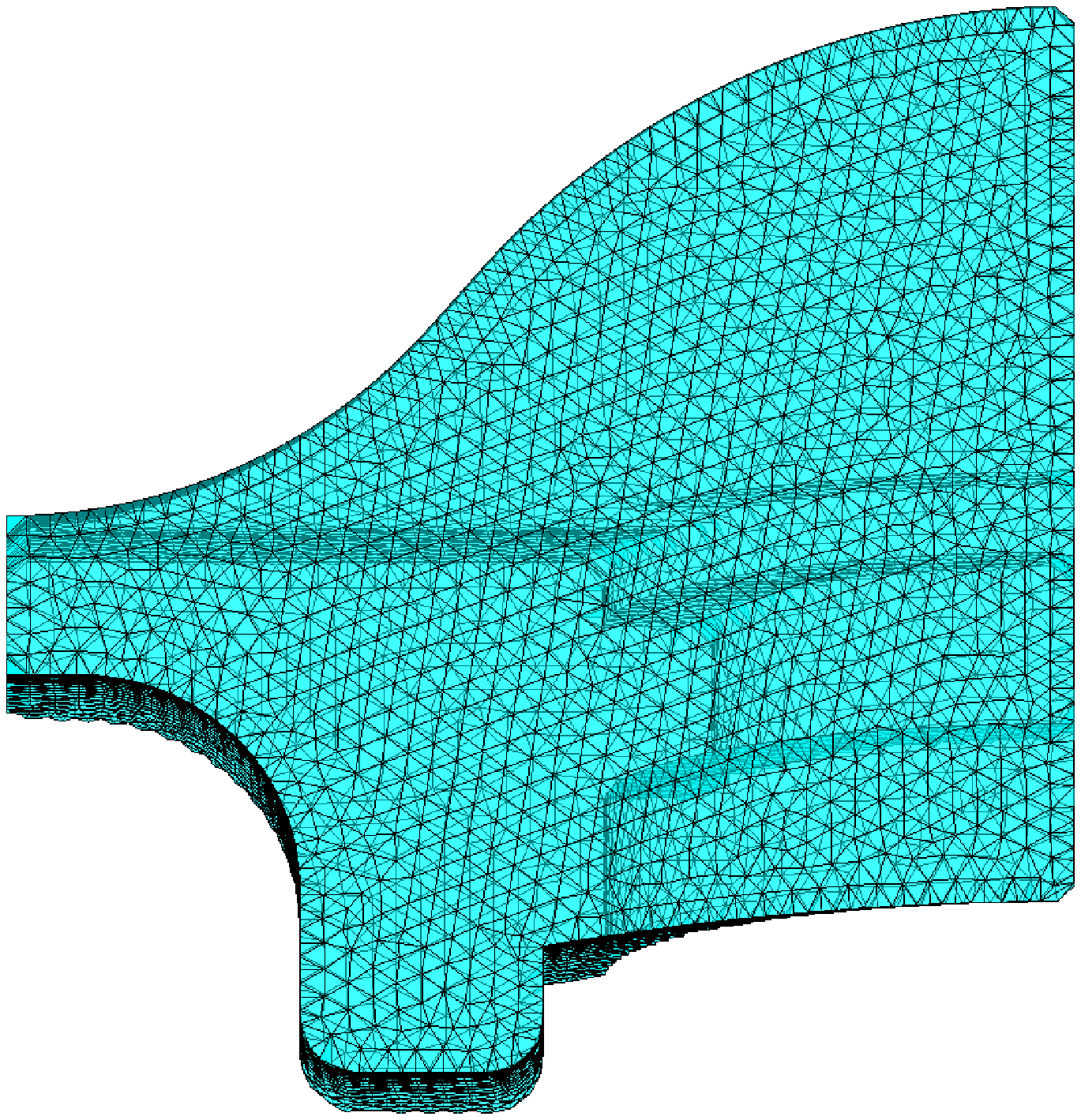}
\caption{Fandisk}
    \label{fig:FanDisk_mesh}
\end{subfigure}
\begin{subfigure}[b]{0.23\textwidth}
\center
    \includegraphics[width=\textwidth]{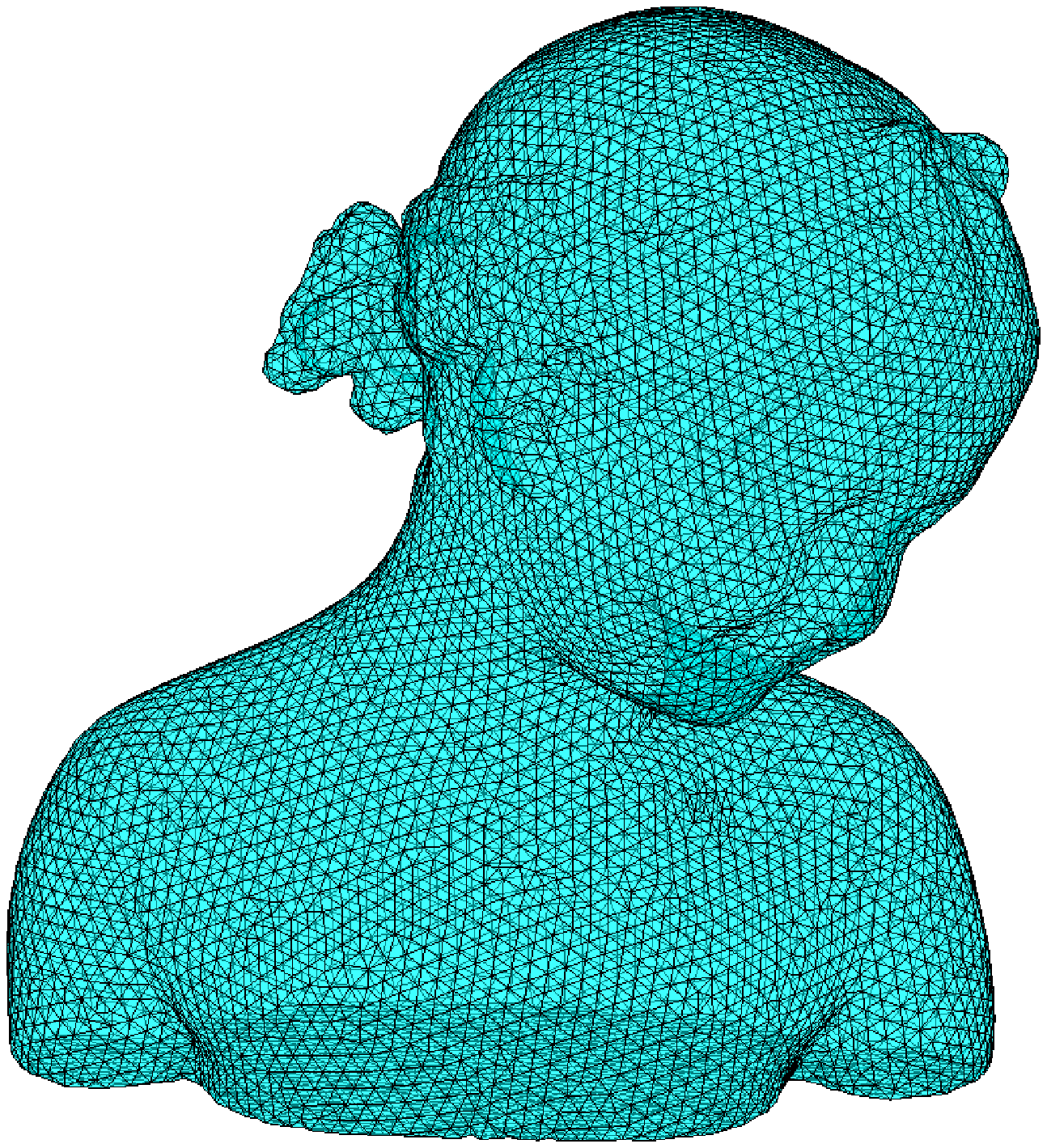}
\caption{Bimba}
    \label{fig:FanDisk_mesh}
\end{subfigure}
\begin{subfigure}[b]{0.23\textwidth}
\center
    \includegraphics[width=\textwidth]{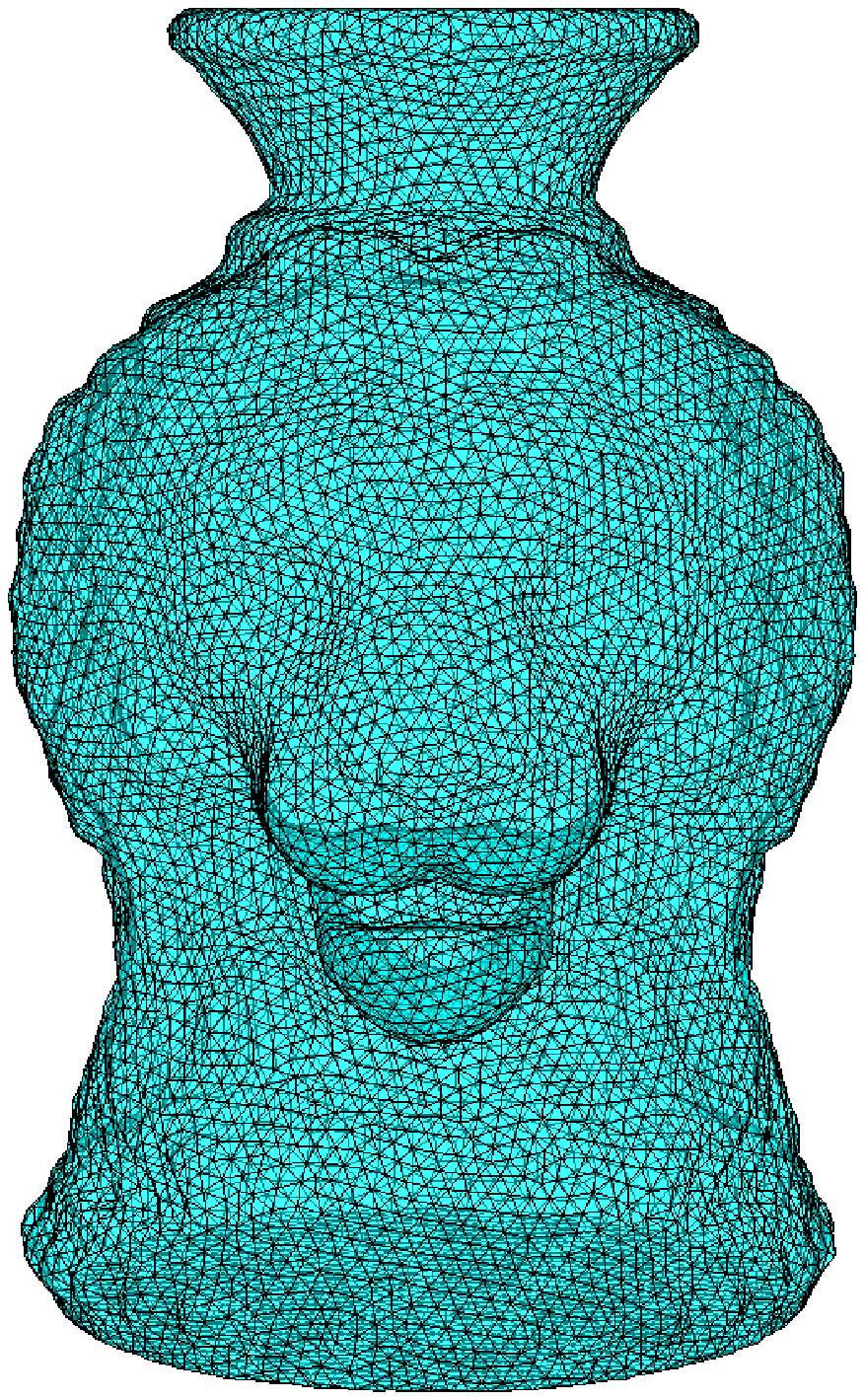}
\caption{Lion}
    \label{fig:Lion_mesh}
\end{subfigure}
\begin{subfigure}[b]{0.26\textwidth}
\center
    \includegraphics[width=\textwidth]{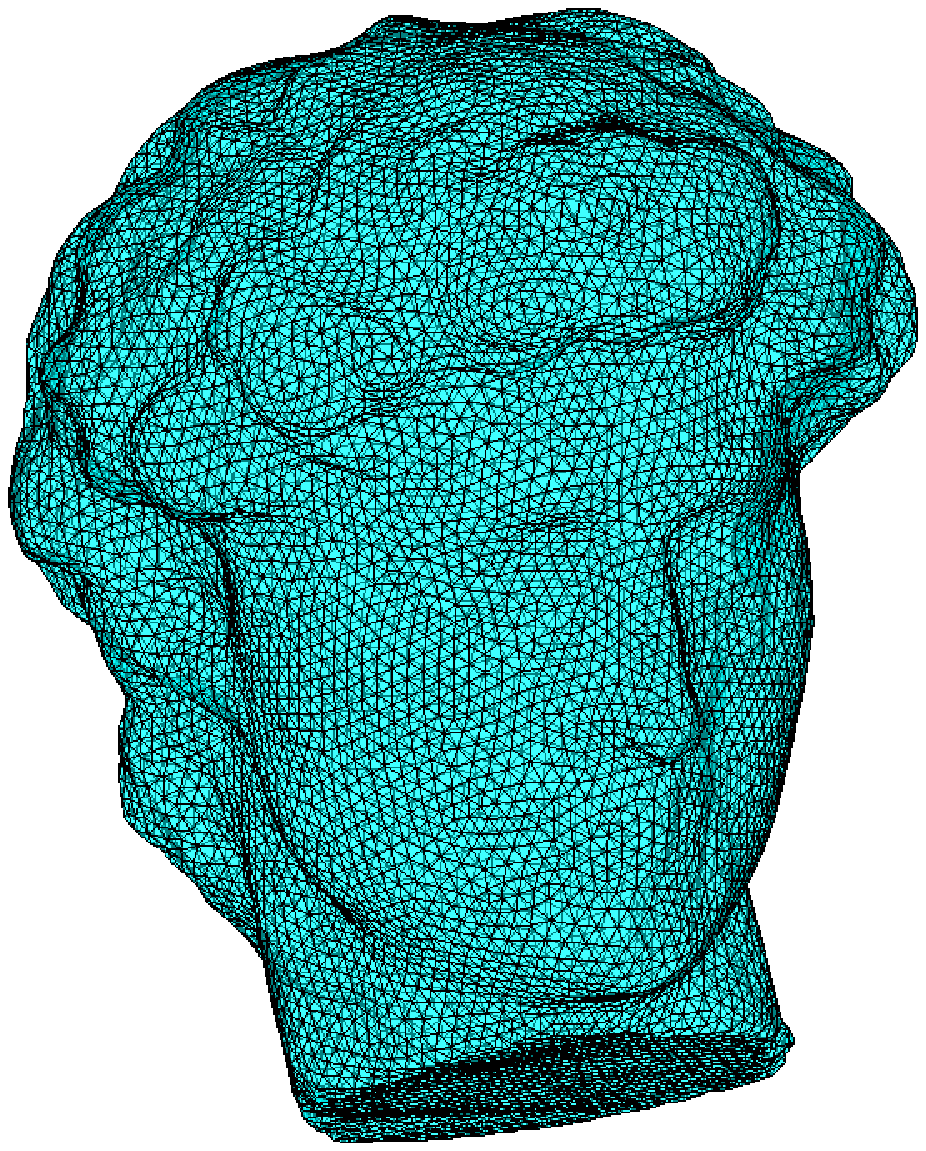}
\caption{David Head}
    \label{fig:Lion_mesh}
\end{subfigure}
\begin{subfigure}[b]{0.23\textwidth}
\center
    \includegraphics[width=\textwidth]{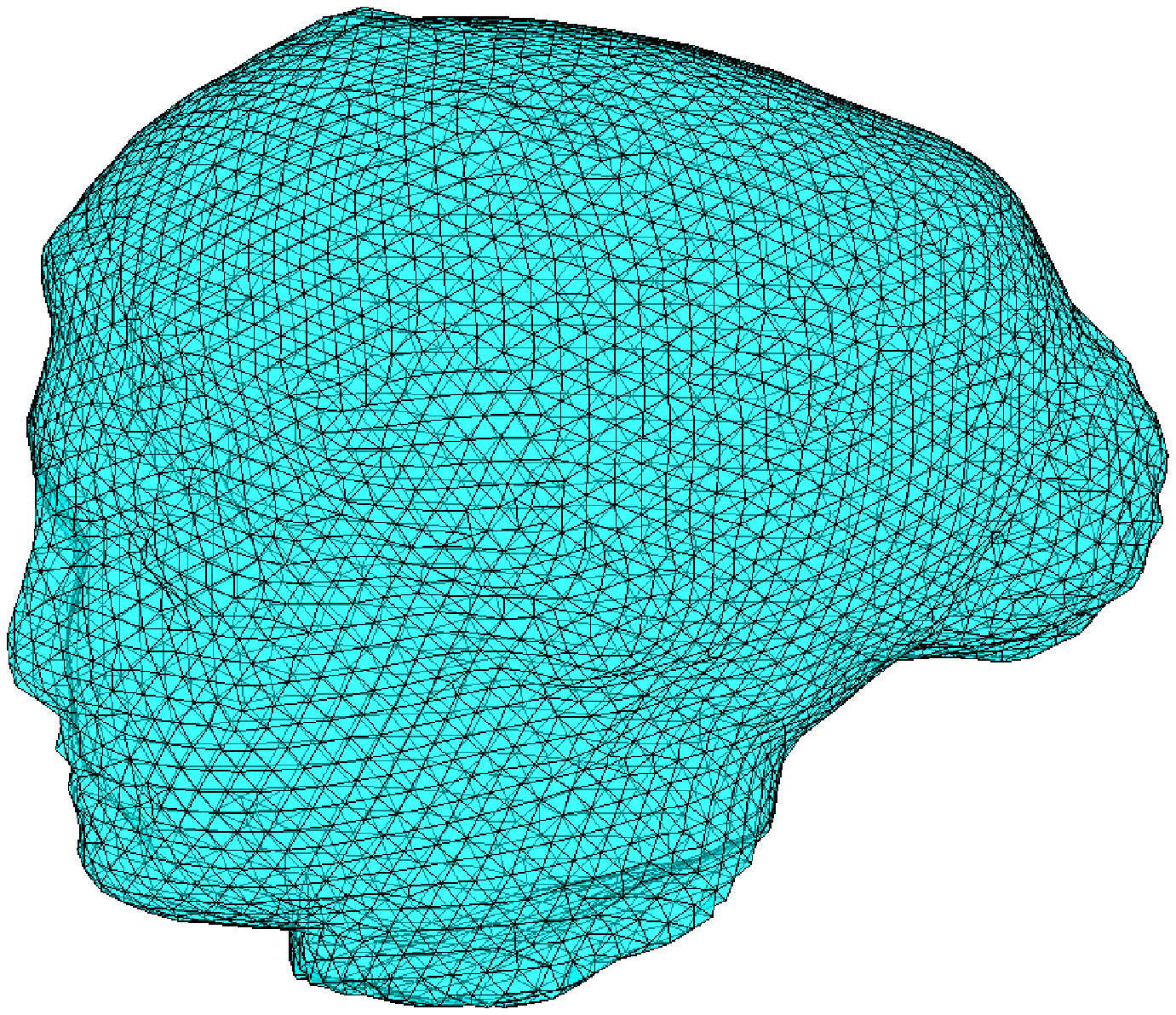}
\caption{Venus}
    \label{fig:Lion_mesh}
\end{subfigure}
\begin{subfigure}[b]{0.22\textwidth}
\center
    \includegraphics[width=\textwidth]{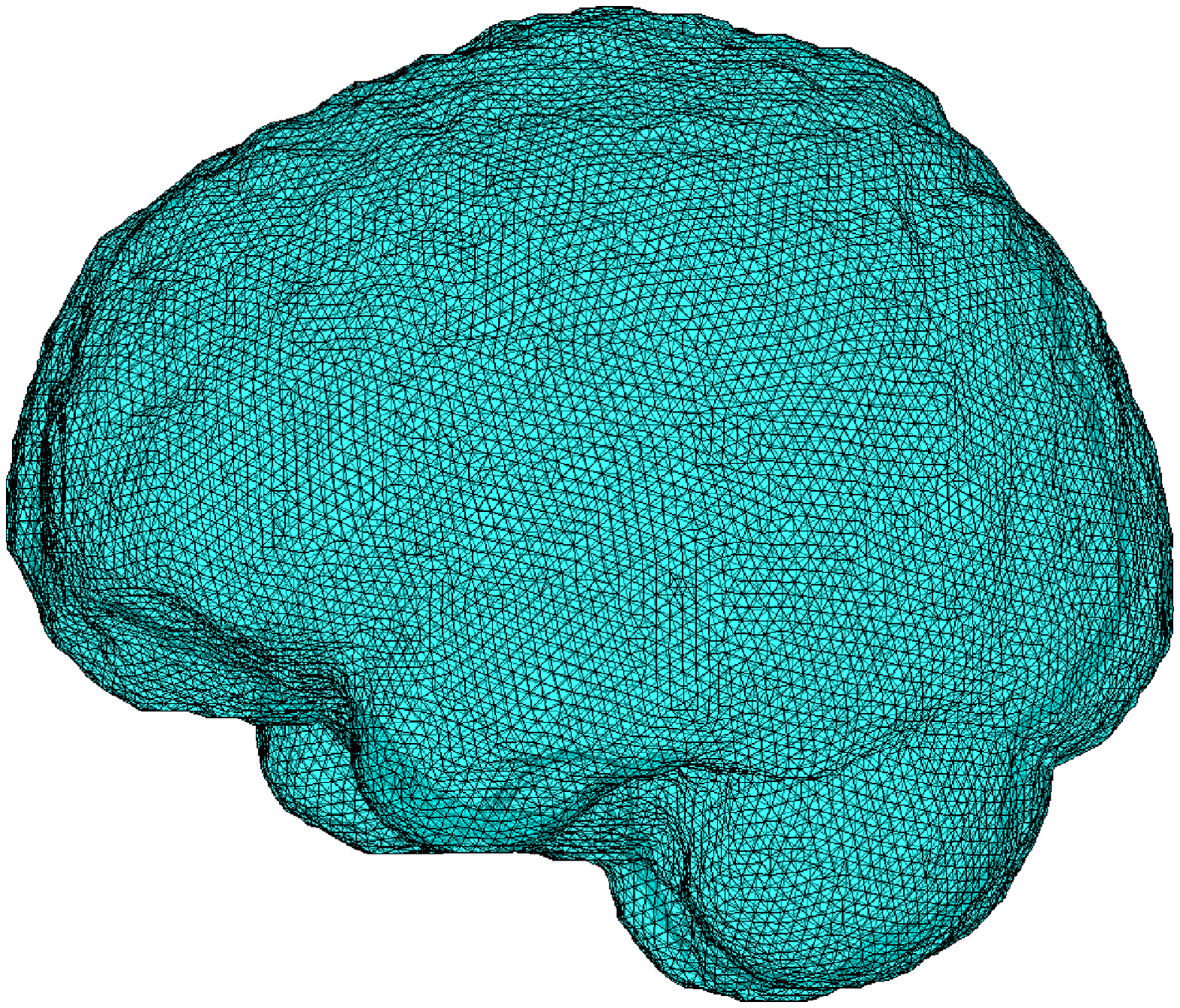}
\caption{Brain}
    \label{fig:Lion_mesh}
\end{subfigure}
\caption{Benchmark mesh models.}
\label{fig:benchmark_mesh}
\end{figure}

In this section, we confirm that assumption \eqref{eq:assumption_convergence_2} of Theorem~\ref{thm:convergence} holds, and we show the convergence of Algorithm~\ref{alg:CEM_defl} for various genus-zero closed surface models, as shown in Figure~\ref{fig:benchmark_mesh} from GitHub~\cite{GitHub}, Gu's website~\cite{GuWeb}, and the BraTS 2021 databases~\cite{BaAS17,BaGB21}. 
In particular, the BraTS 2021 database establishes a challenging platform for brain tumor segmentation by AI techniques and provides 1251 MRI brain images for training. Finally, we give the conformality and the Dirichlet energy computed by the MDEM Algorithm~\ref{alg:CEM_defl}, which are similar to those computed by the DEM Algorithm~\ref{alg1}. 
However, Algorithm \ref{alg:CEM_defl} has a strict mathematical proof of convergence in Theorem \ref{thm:convergence}, while Algorithm \ref{alg1} lacks a proof. 

\begin{table}
\center
\begin{tabular}{l|c|c|c|cccc}
\hline
benchmark & Max Planck & Apple & Fandisk & Bimba 
\\ \hline
$n$ & 34,228 & 40,777 & 40,406 & 34,282 
\\ 
$m_1$ & 325 & 352 & 339 & 204 \\
$m_2$ & 317 & 360 & 358 & 222 \\
$\rho(\gamma^2 \lvert \widehat{A}_2 \rvert A_1)$  &  0.776 & 0.739 & 0.859 & 0.876 \\
\hline \hline 
benchmark & Lion & David Head & Venus & Brain \\ \hline
$n$ &  40,792 & 38,484 & 34,807 & 10,107 \\ 
$m_1$ & 372 & 334 & 330 & 169 \\
$m_2$ & 281 & 316 & 340 & 185 \\
$\rho(\gamma^2 \lvert \widehat{A}_2 \rvert A_1)$ & 0.815 & 0.776 & 0.786 & 0.822 \\
\hline
\end{tabular}
\caption{Numbers of vertices for the triangular meshes in the benchmark models and the associated $\rho(\gamma^2 \lvert \widehat{A}_2 \rvert A_1)$.}
\label{tab:spectral_radius}
\end{table}

From \eqref{eq:assumption_convergence_1}, we define $\eta = \max_{k}\{ \max \{ \eta_1^{(k)}, \eta_2^{(k)}\} \}$ with
\begin{align*}  
      \eta_1^{(k)} = 1 - \min_{i = 1, \dots, m_1} \{\lvert A_1\mathbf{h}_1^{(k)}\rvert_i \}, \quad \eta_2^{(k)} = 1 - \min_{i = 1, \dots, m_2} \{\lvert\widehat{A}_2\mathbf{h}_2^{(k)}\rvert_i \}.
\end{align*}
Then,
$1-\eta\leqslant\lvert A_1\mathbf{h}_1^{(k)}\rvert_i\leqslant 1$ and $1-\eta\leq \lvert\widehat{A}_2\mathbf{h}_2^{(k)}\rvert_i\leq 1$ for all $i$ and $k$. Taking $\gamma = 1/(1 - \eta)^2$, we show in Table~\ref{tab:spectral_radius} that the spectral radii $\rho(\gamma^2 \lvert \widehat{A}_2 \rvert A_1)$ are less than one for all the benchmark models. 
In Figure~\ref{fig:hist_spectral_radius}, we show the histograms of $\rho(\gamma^2 \lvert \widehat{A}_2 \rvert A_1)$ for 1251 brain images from the BraTS 2021 Challenge database, which are all less than 1. All the results in Table~\ref{tab:spectral_radius} and Figure~\ref{fig:hist_spectral_radius} show that the assumptions in \eqref{eq:assumption_convergence} hold. That is, the convergence of the MDEM algorithm with nonequivalence deflation can be theoretically guaranteed, as shown in Theorem~\ref{thm:convergence}.

\begin{figure}
\center
\begin{subfigure}[b]{0.48\textwidth} 
\center
    \includegraphics[width=\textwidth]{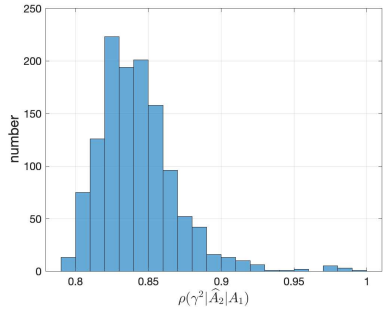}
\end{subfigure}
\caption{Histograms of $\rho(\gamma^2 \lvert \widehat{A}_2 \rvert A_1)$ in Theorem~\ref{thm:convergence} for 1251 brain images.}
\label{fig:hist_spectral_radius}
\end{figure}

In Figures~\ref{fig:MaxPlanck_DavidHead_cong_beh} and \ref{fig:Lion_Brain_conv_beh}, we plot the results of $\| \mathbf{h}_1^{(k+1)} - \mathbf{h}_1^{(k)} \|_2$ for each iteration $k$ to illustrate the convergence behavior of four benchmark models. The results show the typical convergence behavior of Algorithm~\ref{alg:CEM_defl}. 
To validate the asymptotical R-linear convergence of Theorem~\ref{thm:convergence}, we compute 100 iterations of Algorithm~\ref{alg:CEM_defl} for 6 benchmark models and choose $\mathbf{h}_1^{(100)}$ as the converged solution $\mathbf{h}_1^{(\ast)}$. The results $\| \mathbf{h}_1^{(k)}-\mathbf{h}_1^{(*)}\|_\infty^{1/k} := \| \mathbf{h}_1^{(k)}-\mathbf{h}_1^{(100)}\|_\infty^{1/k}$ for each iteration $k$ are plotted in Figure~\ref{fig:R_linear_conv_beh}, which shows that $\| \mathbf{h}_1^{(k)}-\mathbf{h}_1^{(100)}\|_\infty^{1/k} < 1$ for all $k$. 
Now, we give a further discussion of the asymptotically R-linear convergence by increasing the total iteration number $m$ of Algorithm~\ref{alg:CEM_defl}. Since Figure~\ref{fig:conv_beha} shows that all the benchmark models have similar convergence behavior, we take only the model ``Apple'' to demonstrate the convergence $\| \mathbf{h}_1^{(k)}-\mathbf{h}_1^{(*)}\|_\infty^{1/k} := \| \mathbf{h}_1^{(k)}-\mathbf{h}_1^{(m)}\|_\infty^{1/k}$ for $m= 500, 1000, 2000$. 
As shown in Figure~\ref{fig:R_linear_conv_beh}, the value of $\| \mathbf{h}_1^{(k)}-\mathbf{h}_1^{(m)}\|_\infty^{1/k}$ increases as $k$ increases in the first 80 iterations. 
To show the asymptotically R-linear convergence, i.e., $\| \mathbf{h}_s^{(k)}-\mathbf{h}_s^{(*)}\|_\infty^{1/k} \leqslant r_s < 1$ as $k$ becomes sufficiently large, we display the last 200 values of $\| \mathbf{h}_s^{(k)}-\mathbf{h}_s^{(m)}\|_\infty^{1/k}$ for each $m$ in Figure~\ref{fig:conv_beha_Apple_zoomin}. 
The results show that $\| \mathbf{h}_s^{(k)}-\mathbf{h}_s^{(m)}\|_\infty^{1/k}$ increases to nearly 1 and then decreases linearly as $k$ increases. All the values are less than 1, so the convergence of Algorithm~\ref{alg:CEM_defl} is asymptotically R-linear.

\begin{figure}
\center
\begin{subfigure}[b]{0.32\textwidth}
\center
    \includegraphics[width=\textwidth]{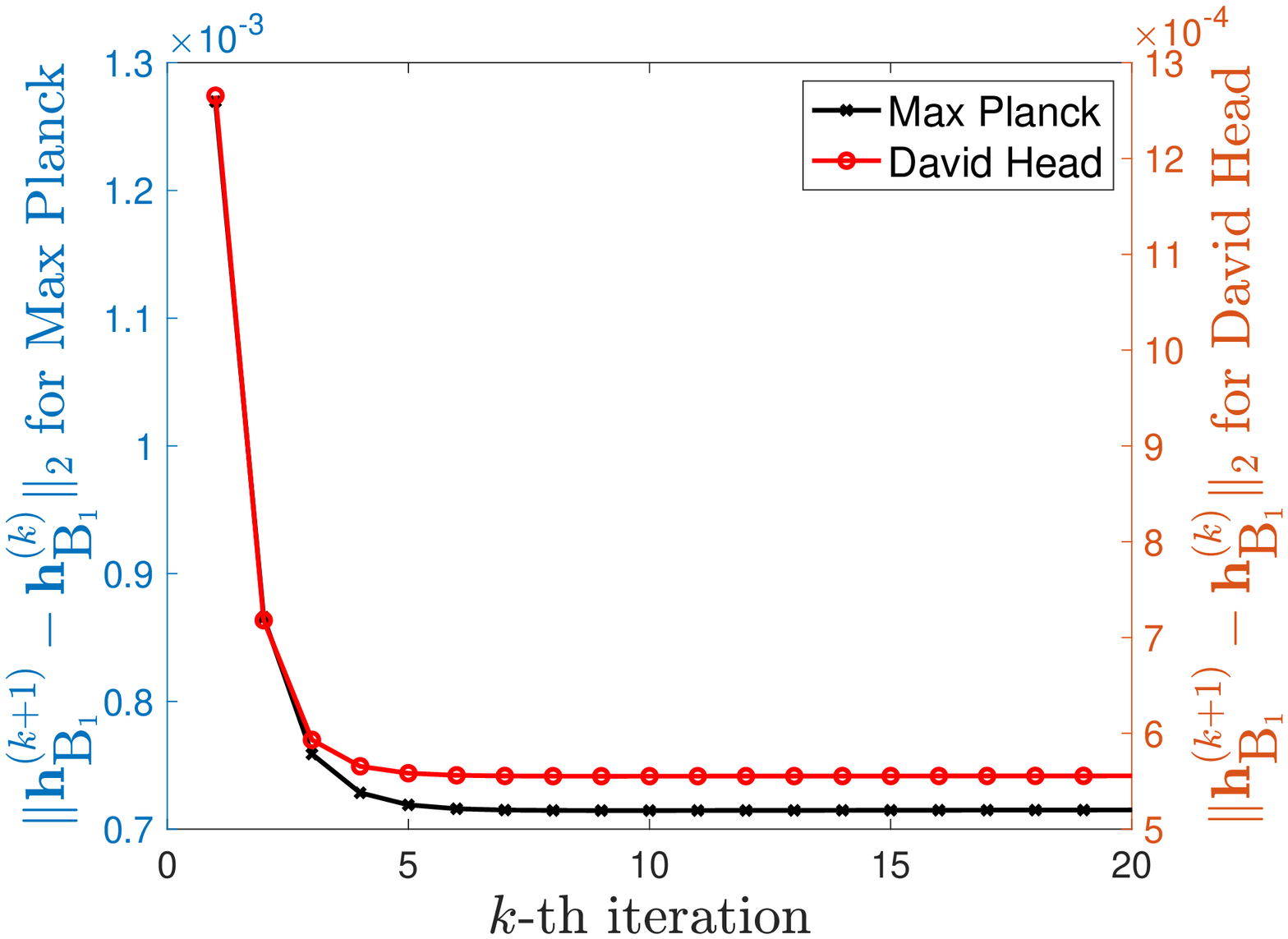}
\caption{}
    \label{fig:MaxPlanck_DavidHead_cong_beh}
\end{subfigure}
\begin{subfigure}[b]{0.32\textwidth}
\center
    \includegraphics[width=\textwidth]{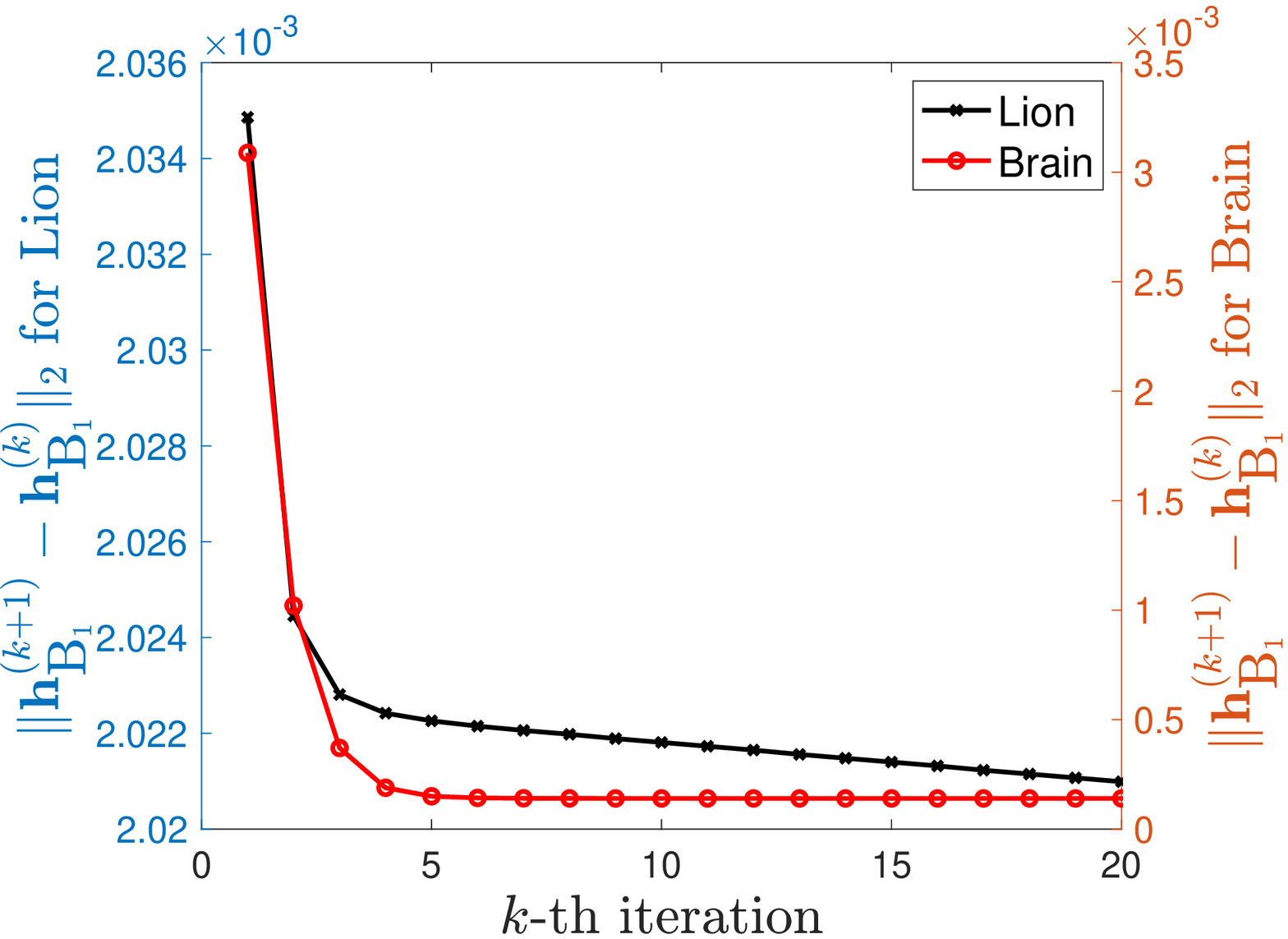}
\caption{}
    \label{fig:Lion_Brain_conv_beh}
\end{subfigure}
\begin{subfigure}[b]{0.32\textwidth}
\center
    \includegraphics[width=\textwidth]{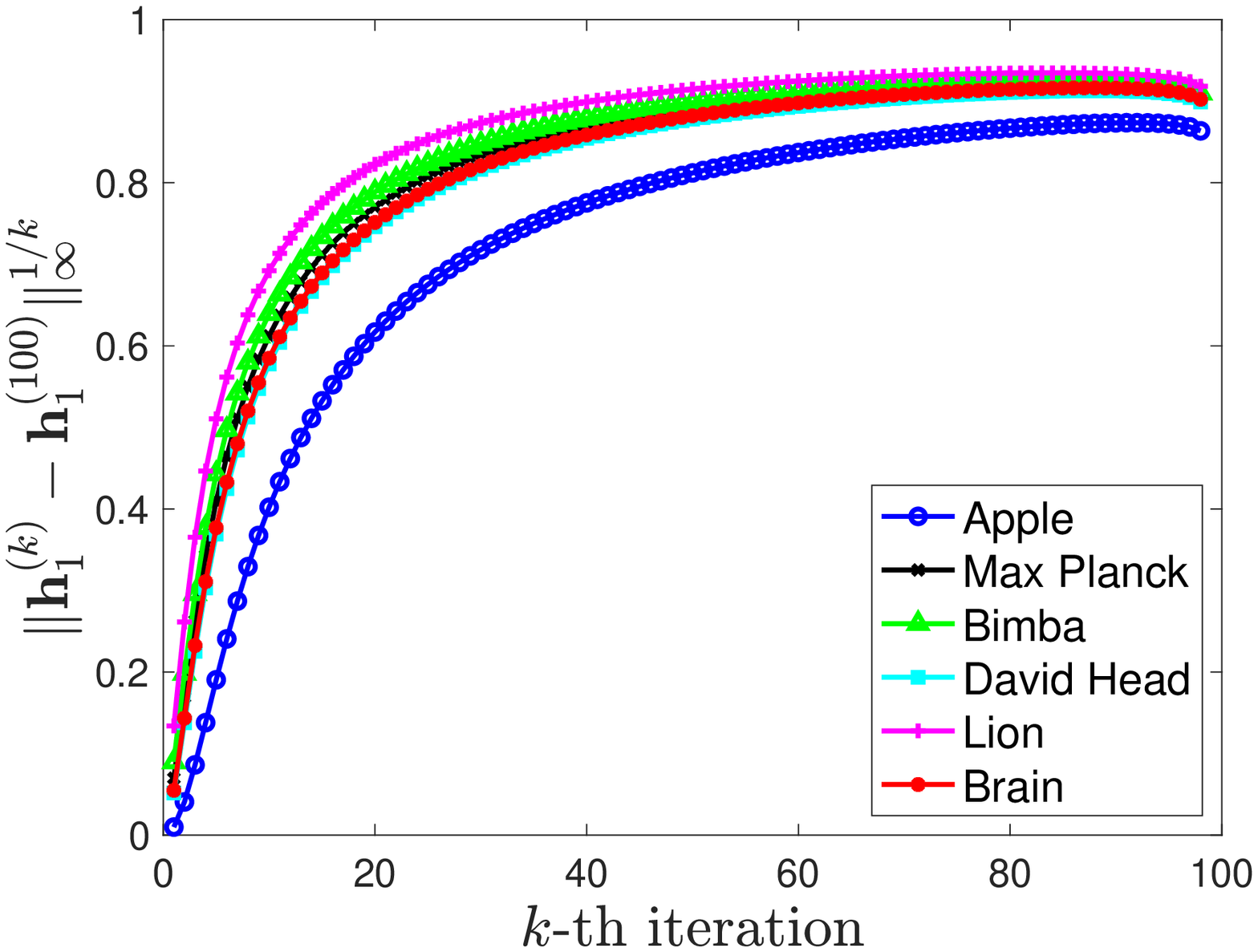}
\caption{}
    \label{fig:R_linear_conv_beh}
\end{subfigure}
\caption{Convergence behavior of the benchmarks (a) ``Max Planck'' and ``David Head'', (b) ``Lion'' and ``Brain'', and (c) R-linear convergence by using Algorithm~\ref{alg:CEM_defl}. }
\label{fig:conv_beha}
\end{figure}

\begin{figure}
\center
\begin{subfigure}[b]{0.48\textwidth}
\center
    \includegraphics[width=\textwidth]{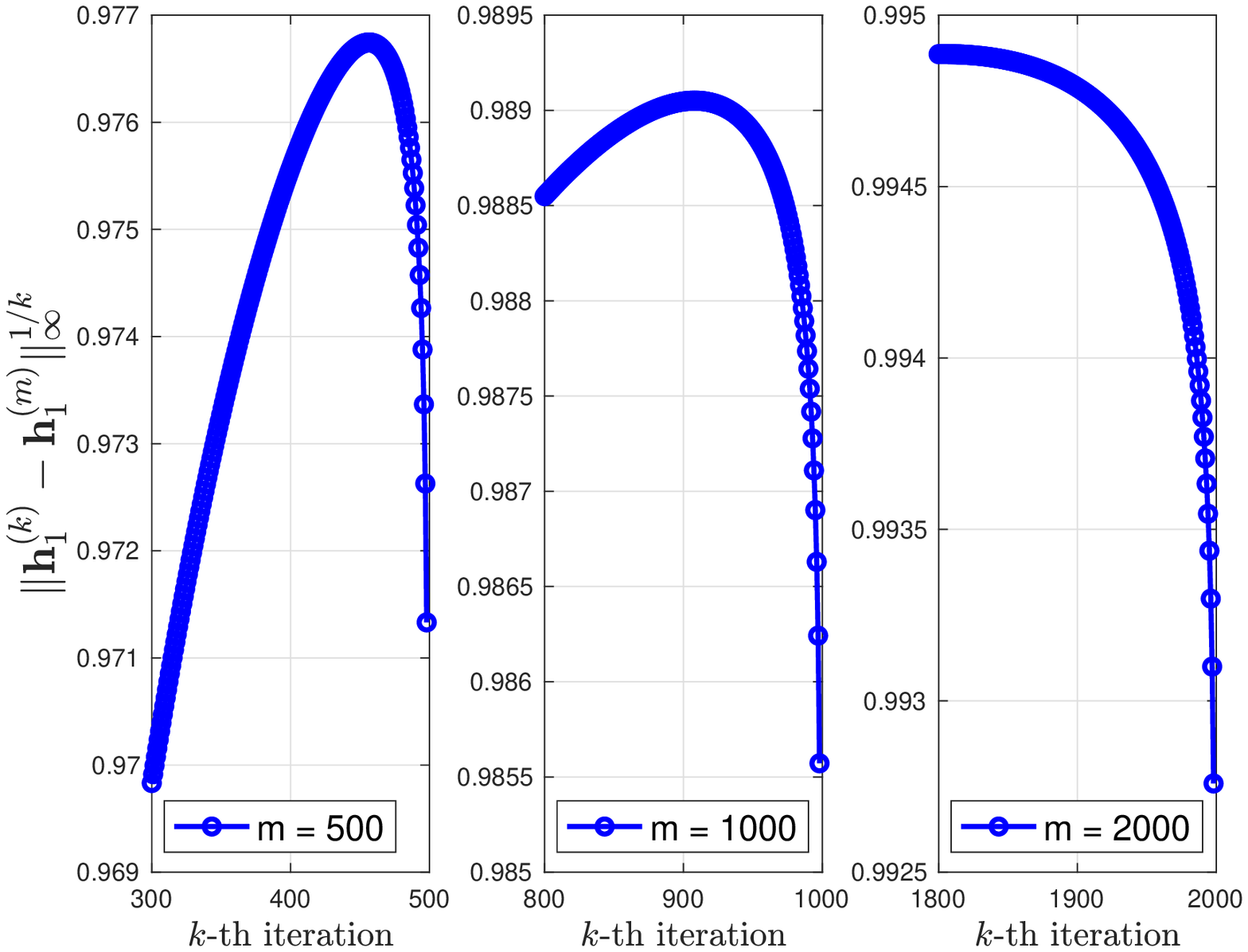}
\caption{$\| \mathbf{h}_1^{(k)}-\mathbf{h}_1^{(m)}\|_\infty^{1/k}$}
    \label{fig:R_linear_conv_Apple_zoomin_h1}
\end{subfigure}
\begin{subfigure}[b]{0.48\textwidth}
\center
    \includegraphics[width=\textwidth]{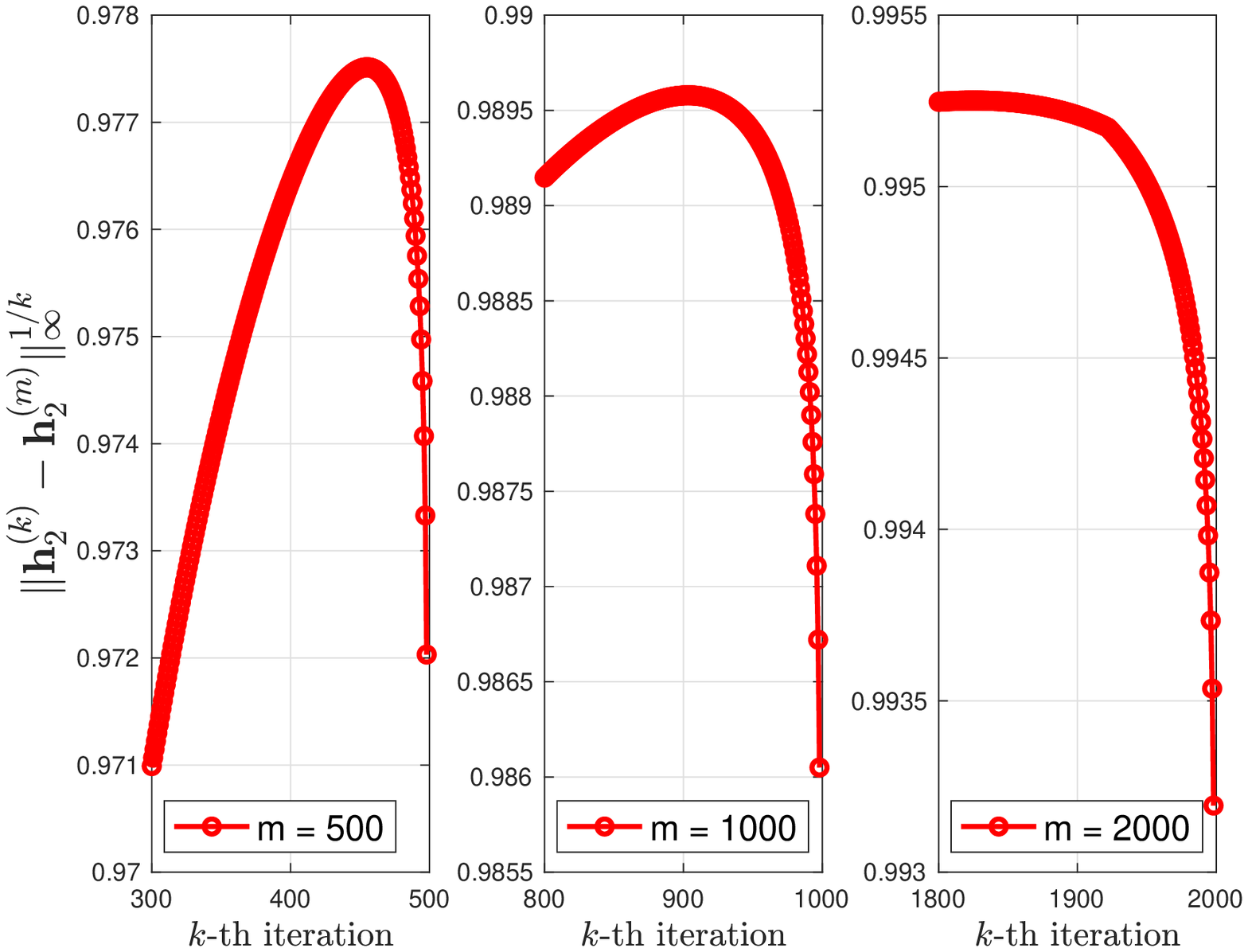}
\caption{$\| \mathbf{h}_2^{(k)}-\mathbf{h}_2^{(m)}\|_\infty^{1/k}$}
    \label{fig:R_linear_conv_Apple_zoomin_h2}
\end{subfigure}
\caption{Enlarged results of $\| \mathbf{h}_s^{(k)}-\mathbf{h}_s^{(m)}\|_\infty^{1/k}$ provided by Algorithm~\ref{alg:CEM_defl} for $m = 500, 1000, 2000$.}
\label{fig:conv_beha_Apple_zoomin}
\end{figure}

Finally, we present the conformality and the Dirichlet energy computed by Algorithms~\ref{alg1} and \ref{alg:CEM_defl}. To show the conformality, we define the relative local angle distortion as $d_{\theta} = \lvert \theta_{\mathcal{S}} - \theta_{\mathbb{S}} \rvert / \theta_{\mathcal{S}}$, where $\theta_{\mathcal{S}}$ and $\theta_{\mathbb{S}}$ are the angles of the triangle meshes in $\mathcal{S}$ and $\mathbb{S}^2$, respectively. In Table~\ref{tab:E_d_d_theta_benchmark}, we see that $E_D(f)$ and the mean and standard deviation (SD) of $\{d_{\theta}\}$ for each benchmark model computed by Algorithm~\ref{alg:CEM_defl} are slightly better than those computed by Algorithm~\ref{alg1}. The spherical conformal parameterization by Algorithm~\ref{alg:CEM_defl} preserves conformality satisfactorily.

Furthermore, in Figures~\ref{fig:mean_d_theta_deflate_1p4} and \ref{fig:STD_defl_1p18}, we plot the histograms of the mean and SD of local angle distortions $\{d_{\theta}\}$ for 1251 brain images by Algorithm~\ref{alg:CEM_defl}. In Figure~\ref{fig:diff_Eng_defl_1p4}, we show the histogram of the difference in the Dirichlet energy $d_{E} = E_{D_2}(f) - E_{D_1}(f)$, where $E_{D_1}(f)$ and $E_{D_2}(f)$ are the Dirichlet energies computed by Algorithms~\ref{alg1} and \ref{alg:CEM_defl}, respectively. Most results in Figure~\ref{fig:diff_Eng_defl_1p4} show that the Dirichlet energy obtained by Algorithm~\ref{alg:CEM_defl} is slightly better than that obtained by Algorithm~\ref{alg1}.

\section{Conclusions} \label{sec:5}

In this paper, we first propose a theoretical foundation of the DEM on $\overline{\mathbb{C}}$ for spherical conformal minimization. By using the calculus of variations, we derive the associated Euler--Lagrange equations, which can be cleverly split into two parts, the homothetic and minimal transformations. The minimal transformation part, which reduces the Dirichlet energy, is meaningful for its contribution to improving conformality, while the homothetic transformation part has no contribution to conformality under the relevant equivalence relation.

On this theoretical basis, we develop an efficient and reliable MDEM algorithm with nonequivalence deflation for the computation of spherical conformal parameterizations from a closed triangular mesh of genus zero to $\mathbb{S}^2$. Furthermore, we prove that the MDEM algorithm has asymptotically R-linear convergence under some mild conditions. 
The numerical results on various models from different benchmarks show that MDEM preserves conformality satisfactorily and is efficient in terms of computational cost, which supports our developed MDEM algorithm.

\begin{table}
\center
\begin{tabular}[t]{l|ccc|ccc}
\hline
  & \multicolumn{3}{c|}{Algorithm~\ref{alg1}} & \multicolumn{3}{c}{Algorithm~\ref{alg:CEM_defl}} \\ \hline 
benchmark   & $E_D(f)$ &  Mean & SD & $E_D(f)$ &  Mean & SD \\
\hline     
Max Planck  & 0.00196 & 0.007 & 0.014 & 0.00196 & 0.007 & 0.014   \\ 
Apple  & 0.00016 & 0.002 & 0.003 &  0.00016 & 0.002 & 0.003 \\
Fandisk  & 0.01507 & 0.019 & 0.037 & 0.01504 & 0.019 & 0.037\\
Bimba       & 0.00290 & 0.012 & 0.015 & 0.00290 & 0.012 & 0.015  \\     
Lion        & 0.00364  & 0.010 & 0.012 & 0.00362 & 0.010 & 0.012  \\    
David Head  & 0.00239  & 0.008  & 0.012  & 0.00239 & 0.008 & 0.012  \\     
Venus     & 0.00090  & 0.005  &  0.008 & 0.00090 & 0.005 & 0.008  \\ 
Brain & 0.00591 & 0.012 & 0.014 & 0.00591 & 0.012 & 0.014 \\
\hline
\end{tabular}
\caption{Dirichlet energy $E_D(f)$ in \eqref{1a} and the mean and SD of the relative  local angle distortions $\{d_{\theta}\}$ computed by Algorithms~\ref{alg1} and \ref{alg:CEM_defl}.} \label{tab:E_d_d_theta_benchmark}
\end{table}

\begin{figure}
\center
\begin{subfigure}[b]{0.32\textwidth}
\center
    \includegraphics[width=\textwidth]{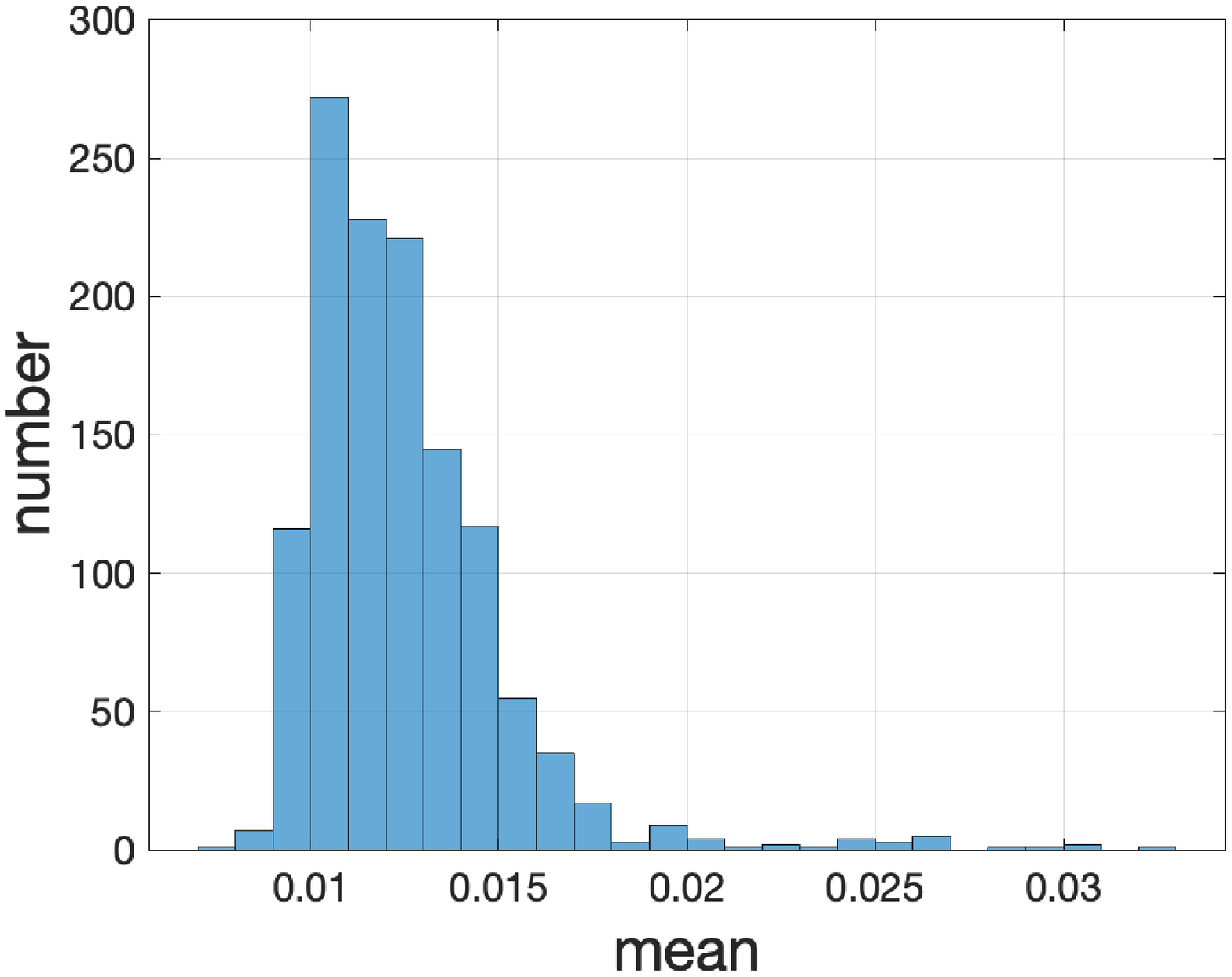}
\caption{Mean}
    \label{fig:mean_d_theta_deflate_1p4}
\end{subfigure}
\begin{subfigure}[b]{0.32\textwidth}
\center
    \includegraphics[width=\textwidth]{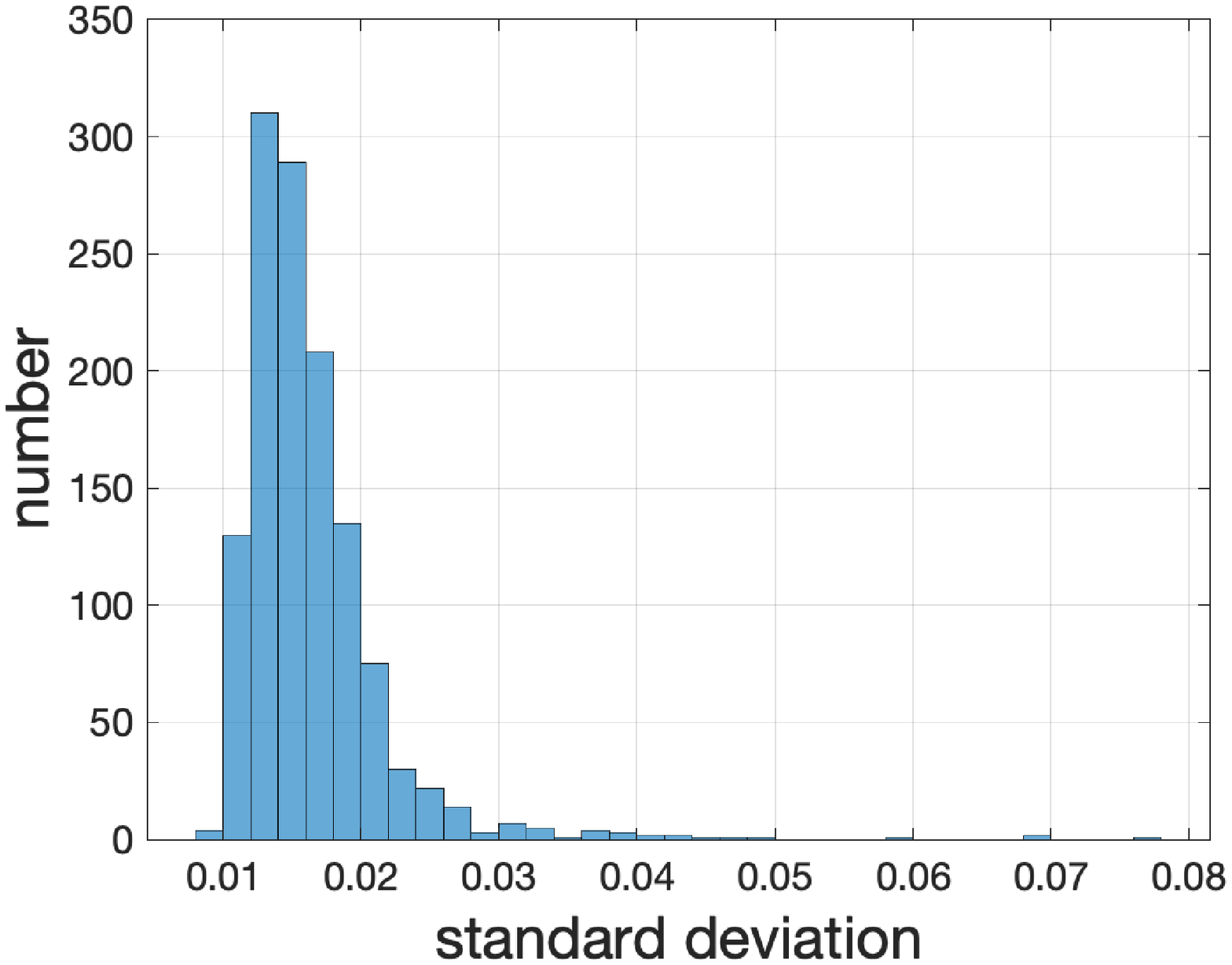}
\caption{SD}
    \label{fig:STD_defl_1p18}
\end{subfigure}
\begin{subfigure}[b]{0.32\textwidth}
\center
    \includegraphics[width=\textwidth]{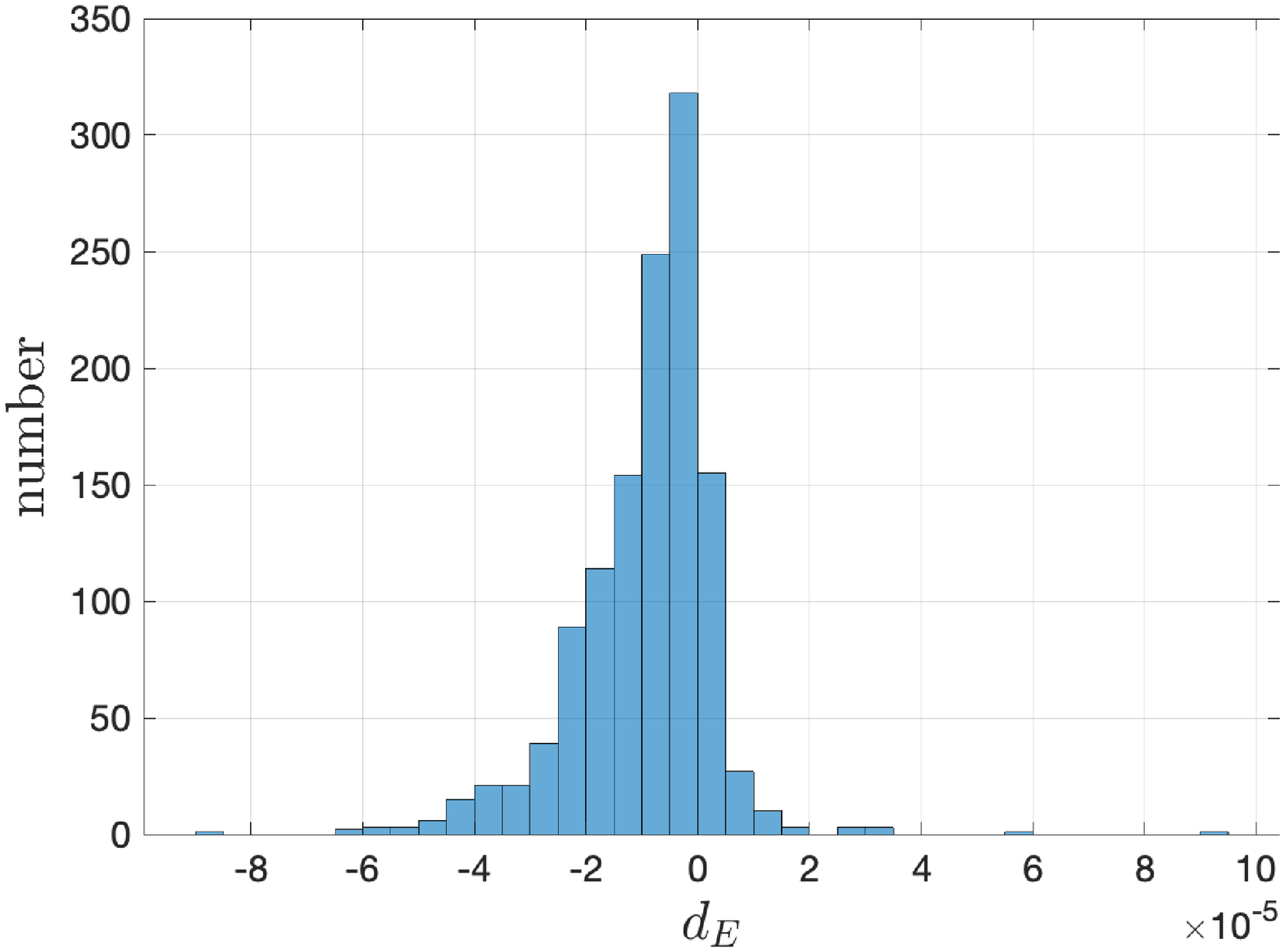}
\caption{$d_{E}$}
    \label{fig:diff_Eng_defl_1p4}
\end{subfigure}
\caption{Histograms of the mean and SD of local angle distortions $\{d_{\theta}\}$ for 1251 brain images by using Algorithm~\ref{alg:CEM_defl} and the difference of Dirichlet energies $d_{E}$ by Algorithms~\ref{alg:CEM_defl} and \ref{alg1}.}
\label{fig:Mean_STD_defl_1p4}
\end{figure}

\bmhead{Acknowledgments}
This work was partially supported by the Ministry of Science and Technology (MoST), the National Center for Theoretical Sciences, the Nanjing Center for Applied Mathematics, and the ST Yau Center in Taiwan. W.-W. Lin, T.-M. Huang, and M.-H. Yueh were partially supported by
MoST 110-2115-M-A49-004-, 110-2115-M-003-012-MY3, and 109-2115-M-003-010-MY2 and 110-2115-M-003-014-, respectively.

\bibliography{reference} 


\begin{thebibliography}{18}
\ifx \bisbn   \undefined \def \bisbn  #1{ISBN #1}\fi
\ifx \binits  \undefined \def \binits#1{#1}\fi
\ifx \bauthor  \undefined \def \bauthor#1{#1}\fi
\ifx \batitle  \undefined \def \batitle#1{#1}\fi
\ifx \bjtitle  \undefined \def \bjtitle#1{#1}\fi
\ifx \bvolume  \undefined \def \bvolume#1{\textbf{#1}}\fi
\ifx \byear  \undefined \def \byear#1{#1}\fi
\ifx \bissue  \undefined \def \bissue#1{#1}\fi
\ifx \bfpage  \undefined \def \bfpage#1{#1}\fi
\ifx \blpage  \undefined \def \blpage #1{#1}\fi
\ifx \burl  \undefined \def \burl#1{\textsf{#1}}\fi
\ifx \doiurl  \undefined \def \doiurl#1{\url{https://doi.org/#1}}\fi
\ifx \betal  \undefined \def \betal{\textit{et al.}}\fi
\ifx \binstitute  \undefined \def \binstitute#1{#1}\fi
\ifx \binstitutionaled  \undefined \def \binstitutionaled#1{#1}\fi
\ifx \bctitle  \undefined \def \bctitle#1{#1}\fi
\ifx \beditor  \undefined \def \beditor#1{#1}\fi
\ifx \bpublisher  \undefined \def \bpublisher#1{#1}\fi
\ifx \bbtitle  \undefined \def \bbtitle#1{#1}\fi
\ifx \bedition  \undefined \def \bedition#1{#1}\fi
\ifx \bseriesno  \undefined \def \bseriesno#1{#1}\fi
\ifx \blocation  \undefined \def \blocation#1{#1}\fi
\ifx \bsertitle  \undefined \def \bsertitle#1{#1}\fi
\ifx \bsnm \undefined \def \bsnm#1{#1}\fi
\ifx \bsuffix \undefined \def \bsuffix#1{#1}\fi
\ifx \bparticle \undefined \def \bparticle#1{#1}\fi
\ifx \barticle \undefined \def \barticle#1{#1}\fi
\bibcommenthead
\ifx \bconfdate \undefined \def \bconfdate #1{#1}\fi
\ifx \botherref \undefined \def \botherref #1{#1}\fi
\ifx \url \undefined \def \url#1{\textsf{#1}}\fi
\ifx \bchapter \undefined \def \bchapter#1{#1}\fi
\ifx \bbook \undefined \def \bbook#1{#1}\fi
\ifx \bcomment \undefined \def \bcomment#1{#1}\fi
\ifx \oauthor \undefined \def \oauthor#1{#1}\fi
\ifx \citeauthoryear \undefined \def \citeauthoryear#1{#1}\fi
\ifx \endbibitem  \undefined \def \endbibitem {}\fi
\ifx \bconflocation  \undefined \def \bconflocation#1{#1}\fi
\ifx \arxivurl  \undefined \def \arxivurl#1{\textsf{#1}}\fi
\csname PreBibitemsHook\endcsname

\bibitem{koebe1907}
\begin{barticle}
\bauthor{\bsnm{Koebe}, \binits{P.}}:
\batitle{{\"U}ber die uniformisierung beliebiger analytischer kurven}.
\bjtitle{Nachrichten von der Gesellschaft der Wissenschaften zu G{\"o}ttingen,
  Mathematisch-Physikalische Klasse}
\bvolume{1907},
\bfpage{191}--\blpage{210}
(\byear{1907})
\end{barticle}
\endbibitem

\bibitem{poincare1908}
\begin{barticle}
\bauthor{\bsnm{Poincar{\'e}}, \binits{H.}}:
\batitle{Sur l'uniformisation des fonctions analytiques}.
\bjtitle{Acta Math.}
\bvolume{31},
\bfpage{1}--\blpage{63}
(\byear{1908})
\end{barticle}
\endbibitem

\bibitem{hilbert1935}
\begin{bchapter}
\bauthor{\bsnm{Hilbert}, \binits{D.}}:
\bctitle{Zur theorie der konformen abbildung}.
In: \bbtitle{Dritter Band: Analysis{\textperiodcentered} Grundlagen der
  Mathematik{\textperiodcentered} Physik Verschiedenes: Nebst Einer
  Lebensgeschichte},
pp. \bfpage{73}--\blpage{80}.
\bpublisher{Springer},
\blocation{Verlag Berlin Heidelberg}
(\byear{1935})
\end{bchapter}
\endbibitem

\bibitem{AnHT99}
\begin{barticle}
\bauthor{\bsnm{Angenent}, \binits{S.}},
\bauthor{\bsnm{Haker}, \binits{S.}},
\bauthor{\bsnm{Tannenbaum}, \binits{A.}},
\bauthor{\bsnm{Kikinis}, \binits{R.}}:
\batitle{On the {L}aplace-{B}eltrami operator and brain surface flattening}.
\bjtitle{IEEE Trans. Med. Imaging}
\bvolume{18}(\bissue{8}),
\bfpage{700}--\blpage{711}
(\byear{1999}).
\doiurl{10.1109/42.796283}
\end{barticle}
\endbibitem

\bibitem{GuWC04}
\begin{barticle}
\bauthor{\bsnm{Gu}, \binits{X.}},
\bauthor{\bsnm{Wang}, \binits{Y.}},
\bauthor{\bsnm{Chan}, \binits{T.F.}},
\bauthor{\bsnm{Thompson}, \binits{P.M.}},
\bauthor{\bsnm{Yau}, \binits{S.-T.}}:
\batitle{Genus zero surface conformal mapping and its application to brain
  surface mapping}.
\bjtitle{IEEE Trans. Med. Imaging}
\bvolume{23}(\bissue{8}),
\bfpage{949}--\blpage{958}
(\byear{2004}).
\doiurl{10.1109/TMI.2004.831226}
\end{barticle}
\endbibitem

\bibitem{HuGH14}
\begin{barticle}
\bauthor{\bsnm{Huang}, \binits{W.-Q.}},
\bauthor{\bsnm{Gu}, \binits{X.D.}},
\bauthor{\bsnm{Huang}, \binits{T.-M.}},
\bauthor{\bsnm{Lin}, \binits{S.-S.}},
\bauthor{\bsnm{Lin}, \binits{W.-W.}},
\bauthor{\bsnm{Yau}, \binits{S.-T.}}:
\batitle{High performance computing for spherical conformal and {R}iemann
  mappings}.
\bjtitle{Geom. Imag. Comput.}
\bvolume{1}(\bissue{2}),
\bfpage{223}--\blpage{258}
(\byear{2014}).
\doiurl{10.4310/GIC.2014.v1.n2.a2}
\end{barticle}
\endbibitem

\bibitem{ChLL15}
\begin{barticle}
\bauthor{\bsnm{Choi}, \binits{P.T.}},
\bauthor{\bsnm{Lam}, \binits{K.C.}},
\bauthor{\bsnm{Lui}, \binits{L.M.}}:
\batitle{{FLASH}: Fast landmark aligned spherical harmonic parameterization for
  genus-0 closed brain surfaces}.
\bjtitle{SIAM J. Imaging Sci.}
\bvolume{8}(\bissue{1}),
\bfpage{67}--\blpage{94}
(\byear{2015}).
\doiurl{10.1137/130950008}
\end{barticle}
\endbibitem

\bibitem{ChLG20}
\begin{barticle}
\bauthor{\bsnm{Choi}, \binits{G.P.T.}},
\bauthor{\bsnm{Leung-Liu}, \binits{Y.}},
\bauthor{\bsnm{Gu}, \binits{X.}},
\bauthor{\bsnm{Lui}, \binits{L.M.}}:
\batitle{Parallelizable global conformal parameterization of simply-connected
  surfaces via partial welding}.
\bjtitle{SIAM J. Imaging Sci.}
\bvolume{13}(\bissue{3}),
\bfpage{1049}--\blpage{1083}
(\byear{2020}).
\doiurl{10.1137/19M125337X}
\end{barticle}
\endbibitem

\bibitem{YuLW17}
\begin{barticle}
\bauthor{\bsnm{Yueh}, \binits{M.-H.}},
\bauthor{\bsnm{Lin}, \binits{W.-W.}},
\bauthor{\bsnm{Wu}, \binits{C.-T.}},
\bauthor{\bsnm{Yau}, \binits{S.-T.}}:
\batitle{An efficient energy minimization for conformal parameterizations}.
\bjtitle{J. Sci. Comput.}
\bvolume{73}(\bissue{1}),
\bfpage{203}--\blpage{227}
(\byear{2017}).
\doiurl{10.1007/s10915-017-0414-y}
\end{barticle}
\endbibitem

\bibitem{YuLL19}
\begin{barticle}
\bauthor{\bsnm{Yueh}, \binits{M.-H.}},
\bauthor{\bsnm{Li}, \binits{T.}},
\bauthor{\bsnm{Lin}, \binits{W.-W.}},
\bauthor{\bsnm{Yau}, \binits{S.-T.}}:
\batitle{A novel algorithm for volume-preserving parameterizations of
  3-manifolds}.
\bjtitle{SIAM J. Imag. Sci.}
\bvolume{12}(\bissue{2}),
\bfpage{1071}--\blpage{1098}
(\byear{2019}).
\doiurl{10.1137/18M1201184}
\end{barticle}
\endbibitem

\bibitem{BaAS17}
\begin{botherref}
\oauthor{\bsnm{Bakas}, \binits{S.}},
\oauthor{\bsnm{Akbari}, \binits{H.}},
\oauthor{\bsnm{Sotiras}, \binits{A.}},
\oauthor{\bsnm{Bilello}, \binits{M.}},
\oauthor{\bsnm{Rozycki}, \binits{M.}},
\oauthor{\bsnm{Kirby}, \binits{J.S.}},
\oauthor{\bsnm{Freymann}, \binits{J.B.}},
\oauthor{\bsnm{Farahani}, \binits{K.}},
\oauthor{\bsnm{Davatzikos}, \binits{C.}}:
Advancing the cancer genome {A}tlas glioma {MRI} collections with expert
  segmentation labels and radiomic features.
Sci. Data
\textbf{4}(170117)
(2017).
\doiurl{10.1038/sdata.2017.117}
\end{botherref}
\endbibitem

\bibitem{BaGB21}
\begin{botherref}
\oauthor{\bsnm{Baid}, \binits{U.}},
\oauthor{\bsnm{Ghodasara}, \binits{S.}},
\oauthor{\bsnm{Bilello}, \binits{M.}},
\oauthor{\bsnm{Mohan}, \binits{S.}},
\oauthor{\bsnm{Calabrese}, \binits{E.}},
\oauthor{\bsnm{Colak}, \binits{E.}},
\oauthor{\bsnm{Farahani}, \binits{K.}},
\oauthor{\bsnm{Kalpathy-Cramer}, \binits{J.}},
\oauthor{\bsnm{Kitamura}, \binits{F.C.}},
\oauthor{\bsnm{Pati}, \binits{S.}},
\oauthor{\bsnm{Prevedello}, \binits{L.M.}},
\oauthor{\bsnm{Rudie}, \binits{J.D.}},
\oauthor{\bsnm{Sako}, \binits{C.}},
\oauthor{\bsnm{Shinohara}, \binits{R.T.}},
\oauthor{\bsnm{Bergquist}, \binits{T.}},
\oauthor{\bsnm{Chai}, \binits{R.}},
\oauthor{\bsnm{Eddy}, \binits{J.}},
\oauthor{\bsnm{Elliott}, \binits{J.}},
\oauthor{\bsnm{Reade}, \binits{W.}},
\oauthor{\bsnm{Schaffter}, \binits{T.}},
\oauthor{\bsnm{Yu}, \binits{T.}},
\oauthor{\bsnm{Zheng}, \binits{J.}},
\oauthor{\bsnm{Annotators}, \binits{B.}},
\oauthor{\bsnm{Davatzikos}, \binits{C.}},
\oauthor{\bsnm{Mongan}, \binits{J.}},
\oauthor{\bsnm{Hess}, \binits{C.}},
\oauthor{\bsnm{Cha}, \binits{S.}},
\oauthor{\bsnm{Villanueva-Meyer}, \binits{J.}},
\oauthor{\bsnm{Freymann}, \binits{J.B.}},
\oauthor{\bsnm{Kirby}, \binits{J.S.}},
\oauthor{\bsnm{Wiestler}, \binits{B.}},
\oauthor{\bsnm{Crivellaro}, \binits{P.}},
\oauthor{\bsnm{Colen}, \binits{R.R.}},
\oauthor{\bsnm{Kotrotsou}, \binits{A.}},
\oauthor{\bsnm{Marcus}, \binits{D.}},
\oauthor{\bsnm{Milchenko}, \binits{M.}},
\oauthor{\bsnm{Nazeri}, \binits{A.}},
\oauthor{\bsnm{Fathallah-Shaykh}, \binits{H.}},
\oauthor{\bsnm{Wiest}, \binits{R.}},
\oauthor{\bsnm{Jakab}, \binits{A.}},
\oauthor{\bsnm{Weber}, \binits{M.-A.}},
\oauthor{\bsnm{Mahajan}, \binits{A.}},
\oauthor{\bsnm{Menze}, \binits{B.}},
\oauthor{\bsnm{Flanders}, \binits{A.E.}},
\oauthor{\bsnm{Bakas}, \binits{S.}}:
The {RSNA-ASNR-MICCAI} {B}ra{TS} 2021 Benchmark on Brain Tumor Segmentation and
  Radiogenomic Classification
(2021)
\end{botherref}
\endbibitem

\bibitem{jost2008}
\begin{bbook}
\bauthor{\bsnm{Jost}, \binits{J.}},
\bauthor{\bsnm{Jost}, \binits{J.}}:
\bbtitle{Riemannian Geometry and Geometric Analysis}
vol. \bseriesno{42005}.
\bpublisher{Springer},
\blocation{Verlag Berlin Heidelberg}
(\byear{2008})
\end{bbook}
\endbibitem

\bibitem{ahlfors2010}
\begin{bbook}
\bauthor{\bsnm{Ahlfors}, \binits{L.V.}}:
\bbtitle{Conformal Invariants: Topics in Geometric Function Theory}
vol. \bseriesno{371}.
\bpublisher{AMS Chelsea Publishing},
\blocation{USA}
(\byear{2010})
\end{bbook}
\endbibitem

\bibitem{GuYa08}
\begin{bbook}
\bauthor{\bsnm{Gu}, \binits{X.D.}},
\bauthor{\bsnm{Yau}, \binits{S.-T.}}:
\bbtitle{Computational Conformal Geometry}
vol. \bseriesno{1}.
\bpublisher{International Press},
\blocation{Somerville, Massachusetts, USA}
(\byear{2008})
\end{bbook}
\endbibitem

\bibitem{meye:2000}
\begin{bbook}
\bauthor{\bsnm{Meyer}, \binits{C.D.}}:
\bbtitle{Matrix Analysis and Applied Linear Algebra}.
\bpublisher{SIAM},
\blocation{Philadelphia, PA, USA}
(\byear{2000})
\end{bbook}
\endbibitem

\bibitem{GitHub}
\begin{botherref}
{GitHub - alecjacobson/common-3d-test-models: Repository containing common 3D
  test models in original format with original source if known and obj mesh}.
\url{https://github.com/alecjacobson/common-3d-test-models}
(2021)
\end{botherref}
\endbibitem

\bibitem{GuWeb}
\begin{botherref}
{David Xianfeng Gu's Home Page}.
\url{http://www3.cs.stonybrook.edu/~gu/}
(2017)
\end{botherref}
\endbibitem

\end{thebibliography}

\end{document}